\documentclass[11pt]{article}
\usepackage{amsmath,amsfonts,amssymb,amsthm}
\usepackage{enumerate}
\usepackage{caption}
\usepackage{pbox}
\usepackage[pagebackref,colorlinks,citecolor=blue,linkcolor=blue]{hyperref}

\numberwithin{equation}{section}
\theoremstyle{plain}
\newtheorem{theorem}[equation]{Theorem}
\newtheorem{corollary}[equation]{Corollary}
\newtheorem{lemma}[equation]{Lemma}
\newtheorem{proposition}[equation]{Proposition}

\usepackage{geometry}
\theoremstyle{definition}
\newtheorem{definition}[equation]{Definition}

\numberwithin{equation}{section}

\newcommand{\R}{{\mathbb R}}
\newcommand{\N}{{\mathbb N}}

\newcommand{\Om}{\Omega}

\providecommand{\vint}[1]{\mathchoice
          {\mathop{\vrule width 5pt height 3 pt depth -2.5pt
                  \kern -9pt \kern 1pt\intop}\nolimits_{\kern -5pt{#1}}}
          {\mathop{\vrule width 5pt height 3 pt depth -2.6pt
                  \kern -6pt \intop}\nolimits_{\kern -3pt{#1}}}
          {\mathop{\vrule width 5pt height 3 pt depth -2.6pt
                  \kern -6pt \intop}\nolimits_{\kern -3pt{#1}}}
          {\mathop{\vrule width 5pt height 3 pt depth -2.6pt
                  \kern -6pt \intop}\nolimits_{\kern -3pt{#1}}}}

\newcommand{\eps}{\varepsilon}
\newcommand{\loc}{\mathrm{loc}}

\newcommand{\BV}{\mathrm{BV}}

\newcommand{\ch}{\text{\raise 1.3pt \hbox{$\chi$}\kern-0.2pt}}

\DeclareMathOperator{\Mod}{Mod}
\DeclareMathOperator{\capa}{Cap}
\DeclareMathOperator{\rcapa}{cap}

\DeclareMathOperator{\dist}{dist}
\DeclareMathOperator{\diam}{diam}

\DeclareMathOperator{\Lip}{Lip}
\DeclareMathOperator*{\esssup}{ess\,sup}

\DeclareMathOperator{\supp}{spt}

\begin{document}
\title{A new Federer-type characterization\\
 of sets of finite perimeter
in metric spaces
\footnote{
\hfill \break {\it Keywords\,}: set of finite perimeter, Federer's characterization,
measure-theoretic boundary, lower density, metric measure space, function of least gradient
}}
\author{Panu Lahti}
\maketitle

\begin{abstract}
Federer's characterization states that a set $E\subset \R^n$ is of finite perimeter
if and only if $\mathcal H^{n-1}(\partial^*E)<\infty$. Here the measure-theoretic
boundary $\partial^*E$ consists of those points where both $E$ and its complement have
positive upper density.
We show that the characterization remains true if
$\partial^*E$ is replaced by a smaller boundary
consisting of those points where the \emph{lower} densities of both $E$
and its complement are at least a given number.
This result is new even in Euclidean spaces but we prove
it in a more general complete metric space that is equipped
with a doubling measure and supports a Poincar\'e inequality.
\end{abstract}

\section{Introduction}

Federer's \cite{Fed} characterization of sets of finite perimeter
states that a set $E\subset \R^n$ is of finite perimeter if and only if
$\mathcal H^{n-1}(\partial^*E)<\infty$,
where $\mathcal H^{n-1}$ is the $n-1$-dimensional Hausdorff measure and
$\partial^*E$ is the measure-theoretic boundary;
see Section \ref{sec:preliminaries} for definitions.
A similar characterization holds also in the abstract setting of complete metric spaces
$(X,d,\mu)$
that are equipped with a doubling measure $\mu$ and support a Poincar\'e inequality;
in such spaces one replaces the $n-1$-dimensional Hausdorff measure
with the \emph{codimension one} Hausdorff measure $\mathcal H$.
 The ``only if'' direction of the characterization was
 shown in metric spaces by Ambrosio \cite{A1}, and the ``if'' direction was recently
 shown by the author \cite{L-Fedchar}.

Federer also showed that if a set $E\subset \R^n$ is of finite perimeter, 
then $\mathcal H^{n-1}(\partial^*E\setminus \Sigma_{1/2}E)=0$,
where the boundary $\Sigma_{1/2}E$ consists of those points where both
$E$ and its complement have density exactly $1/2$.
In metric spaces we similarly have
$\mathcal H(\partial^*E\setminus \Sigma_{\gamma}E)=0$,
where $0<\gamma\le 1/2$ is a suitable constant depending on the space
and the \emph{strong boundary} $\Sigma_{\gamma}E$ is defined by
\[
\Sigma_{\gamma} E:=\left\{x\in X:\, \liminf_{r\to 0}\frac{\mu(B(x,r)\cap E)}{\mu(B(x,r))}\ge \gamma\ \ \textrm{and}\ \ \liminf_{r\to 0}\frac{\mu(B(x,r)\setminus E)}{\mu(B(x,r))}\ge \gamma\right\}.
\]
This raises the natural question of whether the condition
$\mathcal H(\Sigma_{\beta} E)<\infty$
for some $\beta>0$, which appears much weaker than $\mathcal H(\partial^* E)<\infty$,
is already enough to imply that $E$
is of finite perimeter. 
Recently Chleb\'ik \cite{Chl}
posed this question in Euclidean spaces and noted that
the (positive) answer is known only when $n=1$.

In the current paper we show that this characterization does indeed hold in every Euclidean space and even in the much more general metric spaces that we consider.

\begin{theorem}\label{thm:main theorem}
Let $(X,d,\mu)$ be a complete metric space with $\mu$ doubling and supporting
a $(1,1)$-Poincar\'e inequality.
Let $\Om\subset X$ be an open set and let $E\subset X$ be a $\mu$-measurable set
with $\mathcal H(\Sigma_{\beta} E\cap \Om)<\infty$,
where $0<\beta\le 1/2$ only depends on the doubling constant of the measure
and the constants in the Poincar\'e inequality. Then $P(E,\Om)<\infty$.
\end{theorem}

Explicitly, in the Euclidean space $\R^n$ with $n\ge 2$, we can take
(see \eqref{eq:choice of beta in Euclidean space})
\[
\beta=  \frac{n^{13n/2}}{2^{26n^2+64n+15}\omega_n^{13}},
\]
where $\omega_n$ is the volume of the Euclidean unit ball.

Our strategy is to show that if $\mathcal H(\Sigma_{\beta}E\cap \Om)<\infty$,
then $\mathcal H((\partial^*E\setminus \Sigma_{\beta}E)\cap \Om)=0$
and so the result follows from the previously known Federer's characterization.
Our proof consists essentially of two steps.
First in Section \ref{sec:strong boundary points},
we show that for every point in the measure-theoretic boundary $\partial^*E$,
arbitrarily close there is a point in the strong boundary $\Sigma_{\beta} E$.
Then, after some preliminary results concerning connected components of sets
of finite perimeter as well as functions of least gradient in Sections
\ref{sec:components} and \ref{sec:least gradient},
in Section \ref{sec:constructing a quasiconvex space} we
show that there exists an open set $V$ containing a suitable part of $\Sigma_{\beta}E$ such that
$X\setminus V$ is itself a metric space with rather good properties.
Thus we can apply the first step in this space.
In Section \ref{sec:proof of the main result} we combine the two steps to prove
Theorem \ref{thm:main theorem}.

\paragraph{Acknowledgments.}
The author wishes to thank Nageswari Shanmugalingam for many helpful comments
as well as for discussions on
constructing spaces where the Mazurkiewicz metric agrees with the ordinary one;
Anders Bj\"orn also for discussions on constructing such spaces; and
Olli Saari for discussions on finding strong boundary
points.

\section{Notation and definitions}\label{sec:preliminaries}

In this section we introduce the notation, definitions,
and assumptions that are employed in the paper.

Throughout this paper, $(X,d,\mu)$ is a complete metric space that is equip\-ped
with a metric $d$ and a Borel regular outer measure $\mu$ satisfying
a doubling property, meaning that
there exists a constant $C_d\ge 1$ such that
\[
0<\mu(B(x,2r))\le C_d\mu(B(x,r))<\infty
\]
for every ball $B(x,r):=\{y\in X:\,d(y,x)<r\}$, with $x\in X$ and $r>0$.
Closed balls are denoted by $\overline{B}(x,r):=\{y\in X:\,d(y,x)\le r\}$.
By iterating the doubling condition, we obtain that for every $x\in X$ and $y\in B(x,R)$ with $0<r\le R<\infty$, we have
\begin{equation}\label{eq:homogenous dimension}
\frac{\mu(B(y,r))}{\mu(B(x,R))}\ge \frac{1}{C_d^2}\left(\frac{r}{R}\right)^{s},
\end{equation}
where $s>1$ only depends on the doubling constant $C_d$.
Given a ball $B=B(x,r)$ and $\beta>0$, we sometimes abbreviate $\beta B:=B(x,\beta r)$;
note that in a metric space, a ball (as a set) does not necessarily have a unique center point and radius, but these will be prescribed for all the balls
that we consider.

We assume that $X$ consists of at least $2$ points.
When we want to state that a constant $C$
depends on the parameters $a,b, \ldots$, we write $C=C(a,b,\ldots)$.
When a property holds outside a set of $\mu$-measure zero, we say that it holds
almost everywhere, abbreviated a.e.

All functions defined on $X$ or its subsets will take values in $[-\infty,\infty]$.
As a complete metric space equipped with a doubling measure, $X$ is proper,
that is, closed and bounded sets are compact.
Since $X$ is proper, for any open set $\Omega\subset X$
we define $L_{\loc}^1(\Omega)$ to be the space of
functions that are in $L^1(\Om')$ for every open $\Omega'\Subset\Omega$.
Here $\Omega'\Subset\Omega$ means that $\overline{\Omega'}$ is a
compact subset of $\Omega$.
Other local spaces of functions are defined analogously.

For any set $A\subset X$ and $0<R<\infty$, the restricted Hausdorff content
of codimension one is defined by
\[
\mathcal{H}_{R}(A):=\inf\left\{ \sum_{j\in I}
\frac{\mu(B(x_{j},r_{j}))}{r_{j}}:\,A\subset
\bigcup_{j\in I}B(x_{j},r_{j}),\,r_{j}\le R,\,I\subset\N\right\}.
\]
The codimension one Hausdorff measure of $A\subset X$ is then defined by
\[
\mathcal{H}(A):=\lim_{R\rightarrow 0}\mathcal{H}_{R}(A).
\]
In the Euclidean space $\R^n$ (equipped with the Euclidean metric and the $n$-dimensional Lebesgue measure) this is comparable to the $n-1$-dimensional
Hausdorff measure.

By a curve we mean a rectifiable continuous mapping from a compact interval of the real line into $X$.
The length of a curve $\gamma$
is denoted by $\ell_{\gamma}$. We will assume every curve to be parametrized
by arc-length, which can always be done (see e.g. \cite[Theorem~3.2]{Hj}).
A nonnegative Borel function $g$ on $X$ is an upper gradient 
of a function $u$
on $X$ if for all nonconstant curves $\gamma$, we have
\begin{equation}\label{eq:definition of upper gradient}
|u(x)-u(y)|\le \int_{\gamma} g\,ds:=\int_0^{\ell_{\gamma}} g(\gamma(s))\,ds,
\end{equation}
where $x$ and $y$ are the end points of $\gamma$.
We interpret $|u(x)-u(y)|=\infty$ whenever  
at least one of $|u(x)|$, $|u(y)|$ is infinite.
Upper gradients were originally introduced in \cite{HK}.

The $1$-modulus of a family of curves $\Gamma$ is defined by
\[
\Mod_{1}(\Gamma):=\inf\int_{X}\rho\, d\mu
\]
where the infimum is taken over all nonnegative Borel functions $\rho$
such that $\int_{\gamma}\rho\,ds\ge 1$ for every curve $\gamma\in\Gamma$.
A property is said to hold for $1$-a.e. curve
if it fails only for a curve family with zero $1$-modulus. 
If $g$ is a nonnegative $\mu$-measurable function on $X$
and (\ref{eq:definition of upper gradient}) holds for $1$-a.e. curve,
we say that $g$ is a $1$-weak upper gradient of $u$. 
By only considering curves $\gamma$ in a set $A\subset X$,
we can talk about a function $g$ being a ($1$-weak) upper gradient of $u$ in $A$.\label{curve discussion}

Given an open set $\Om\subset X$, we let
\[
\Vert u\Vert_{N^{1,1}(\Om)}:=\Vert u\Vert_{L^1(\Om)}+\inf \Vert g\Vert_{L^1(\Om)},
\]
where the infimum is taken over all upper gradients $g$ of $u$ in $\Om$.
Then we define the Newton-Sobolev space
\[
N^{1,1}(\Om):=\{u:\|u\|_{N^{1,1}(\Om)}<\infty\}.
\]
In $\R^n$ this coincides, up to a choice of pointwise representatives,
with the usual Sobolev space $W^{1,1}(\Om)$; this is shown in
Theorem 4.5 of \cite{S}, where the Newton-Sobolev space was originally
introduced.

We understand Newton-Sobolev functions to be defined at every point $x\in \Om$
(even though $\Vert \cdot\Vert_{N^{1,1}(\Om)}$ is then only a seminorm).
It is known that for every $u\in N_{\loc}^{1,1}(\Om)$ there exists a minimal $1$-weak
upper gradient of $u$ in $\Om$, always denoted by $g_{u}$, satisfying $g_{u}\le g$ 
a.e. in $\Om$ for any other  $1$-weak upper gradient $g\in L_{\loc}^{1}(\Om)$
of $u$ in $\Om$, see \cite[Theorem 2.25]{BB}.
In $\R^n$, the minimal $1$-weak upper gradient coincides (a.e.) with $|\nabla u|$,
see \cite[Corollary A.4]{BB}.

We will assume throughout the paper that $X$ supports a $(1,1)$-Poincar\'e inequality,
meaning that there exist constants $C_P\ge 1$ and $\lambda \ge 1$ such that for every
ball $B(x,r)$, every $u\in L^1_{\loc}(X)$,
and every upper gradient $g$ of $u$,
we have
\[
\vint{B(x,r)}|u-u_{B(x,r)}|\, d\mu 
\le C_P r\vint{B(x,\lambda r)}g\,d\mu,
\]
where 
\[
u_{B(x,r)}:=\vint{B(x,r)}u\,d\mu :=\frac 1{\mu(B(x,r))}\int_{B(x,r)}u\,d\mu.
\]

As \label{quasiconvex and geodesic}a complete metric space equipped with a doubling measure and supporting a Poincar\'e
inequality, $X$ is \emph{quasiconvex}, meaning that for every
pair of points $x,y\in X$ there is a curve $\gamma$ with $\gamma(0)=x$,
$\gamma(\ell_{\gamma})=y$, and $\ell_{\gamma}\le Cd(x,y)$,
where $C$ is a constant and
only depends on $C_d$ and $C_P$, see e.g. \cite[Theorem 4.32]{BB}.
Thus a biLipschitz change in the metric gives a geodesic space
(see \cite[Section 4.7]{BB}).
Since Theorem \ref{thm:main theorem}
is easily seen to be invariant under such a biLipschitz
change in the metric, we can assume that $X$ is geodesic.
By \cite[Theorem 4.39]{BB}, in the Poincar\'e inequality we can now choose
$\lambda=1$.

The $1$-capacity of a set $A\subset X$ is defined by
\[
\capa_1(A):=\inf \Vert u\Vert_{N^{1,1}(X)},
\]
where the infimum is taken over all functions $u\in N^{1,1}(X)$ satisfying
$u\ge 1$ in $A$.
The variational $1$-capacity of a set $A\subset \Om$
with respect to an open set $\Om\subset X$ is defined by
\[
\rcapa_1(A,\Om):=\inf \int_X g_u \,d\mu,
\]
where the infimum is taken over functions $u\in N^{1,1}(X)$ satisfying
$u=0$ in $X\setminus\Om$ and
$u\ge 1$ in $A$, and $g_u$ is the minimal $1$-weak upper gradient of $u$ (in $X$).
By truncation, we see that we can assume $0\le u\le 1$ on $X$.
The variational $1$-capacity is an outer
capacity in the sense that if $A\Subset \Om$, then
\begin{equation}\label{eq:rcapa outer capacity}
\rcapa_{1}(A,\Om)
=\inf_{\substack{V\textrm{ open} \\A\subset V\subset \Om}}\rcapa_{1}(V,\Om);
\end{equation}
see \cite[Theorem 6.19(vii)]{BB}.
For basic properties satisfied by capacities, such as monotonicity and countable subadditivity, see e.g. \cite{BB}.

We say that a set $U\subset X$ is $1$-quasiopen\label{quasiopen}
if for every $\eps>0$ there exists an
open set $G\subset X$ such that $\capa_1(G)<\eps$ and $U\cup G$ is open.

Next we present the definition and basic properties of functions
of bounded variation on metric spaces, following \cite{M}. See also e.g. \cite{AFP, EvGa, Fed, Giu84, Zie89} for the classical 
theory in the Euclidean setting.
Given an open set $\Om\subset X$ and a function $u\in L^1_{\loc}(\Om)$,
we define the total variation of $u$ in $\Om$ by
\[
\|Du\|(\Om):=\inf\left\{\liminf_{i\to\infty}\int_\Om g_{u_i}\,d\mu:\, u_i\in N^{1,1}_{\loc}(\Om),\, u_i\to u\textrm{ in } L^1_{\loc}(\Om)\right\},
\]
where each $g_{u_i}$ is the minimal $1$-weak upper gradient of $u_i$
in $\Om$.
In $\R^n$ this agrees with the usual Euclidean definition involving distributional
derivatives, see e.g. \cite[Proposition 3.6, Theorem 3.9]{AFP}.
(In \cite{M}, local Lipschitz constants were used in place of upper gradients, but the theory
can be developed similarly with either definition.)
We say that a function $u\in L^1(\Om)$ is of bounded variation, 
and denote $u\in\BV(\Om)$, if $\|Du\|(\Om)<\infty$.
For an arbitrary set $A\subset X$, we define
\[
\|Du\|(A):=\inf\{\|Du\|(W):\, A\subset W,\,W\subset X
\text{ is open}\}.
\]

If $u\in L^1_{\loc}(\Om)$ and $\Vert Du\Vert(\Omega)<\infty$,
then $\|Du\|(\cdot)$ is
a Borel regular outer measure on $\Omega$ by \cite[Theorem 3.4]{M}.
A $\mu$-measurable set $E\subset X$ is said to be of finite perimeter if $\|D\ch_E\|(X)<\infty$, where $\ch_E$ is the characteristic function of $E$.
The perimeter of $E$ in $\Omega$ is also denoted by
\[
P(E,\Omega):=\|D\ch_E\|(\Omega).
\]
The measure-theoretic interior of a set $E\subset X$ is defined by
\begin{equation}\label{eq:measure theoretic interior}
I_E:=
\left\{x\in X:\,\lim_{r\to 0}\frac{\mu(B(x,r)\setminus E)}{\mu(B(x,r))}=0\right\},
\end{equation}
and the measure-theoretic exterior by
\[
O_E:=
\left\{x\in X:\,\lim_{r\to 0}\frac{\mu(B(x,r)\cap E)}{\mu(B(x,r))}=0\right\}.
\]
The measure-theoretic boundary $\partial^{*}E$ is defined as the set of points
$x\in X$
at which both $E$ and its complement have nonzero upper density, i.e.
\[
\limsup_{r\to 0}\frac{\mu(B(x,r)\cap E)}{\mu(B(x,r))}>0\quad
\textrm{and}\quad\limsup_{r\to 0}\frac{\mu(B(x,r)\setminus E)}{\mu(B(x,r))}>0.
\]
Note that the space $X$ is always partitioned into the disjoint sets
$I_E$, $O_E$, and $\partial^*E$.
By Lebesgue's differentiation theorem (see e.g. \cite[Chapter 1]{Hei}),
for a $\mu$-measurable set $E$ we have $\mu(E\Delta I_E)=0$,
where $\Delta$ is the symmetric difference.

Given a number $0<\gamma\le 1/2$, we also define the strong boundary
\begin{equation}\label{eq:strong boundary}
\Sigma_{\gamma} E:=\left\{x\in X:\, \liminf_{r\to 0}\frac{\mu(B(x,r)\cap E)}{\mu(B(x,r))}\ge \gamma\ \, \textrm{and}\ \, \liminf_{r\to 0}\frac{\mu(B(x,r)\setminus E)}{\mu(B(x,r))}\ge \gamma\right\}.
\end{equation}
For an open set $\Omega\subset X$ and a $\mu$-measurable set $E\subset X$ with
$P(E,\Omega)<\infty$, we have
$\mathcal H((\partial^*E\setminus \Sigma_{\gamma}E)\cap\Om)=0$
for $\gamma \in (0,1/2]$ that only depends on $C_d$ and $C_P$,
see \cite[Theorem 5.4]{A1}.
Moreover, for any Borel set $A\subset\Omega$ we have
\begin{equation}\label{eq:def of theta}
P(E,A)=\int_{\partial^{*}E\cap A}\theta_E\,d\mathcal H,
\end{equation}
where
$\theta_E\colon \Om\to [\alpha,C_d]$ with $\alpha=\alpha(C_d,C_P)>0$, see \cite[Theorem 5.3]{A1} 
and \cite[Theorem 4.6]{AMP}.

The following coarea formula is given in \cite[Proposition 4.2]{M}:
if $\Omega\subset X$ is an open set and $u\in L^1_{\loc}(\Omega)$, then
\begin{equation}\label{eq:coarea}
\|Du\|(\Omega)=\int_{-\infty}^{\infty}P(\{u>t\},\Omega)\,dt,
\end{equation}
where we abbreviate $\{u>t\}:=\{x\in \Om:\,u(x)>t\}$.
If $\Vert Du\Vert(\Om)<\infty$, then \eqref{eq:coarea} holds with $\Om$ replaced by
any Borel set $A\subset \Om$.

We know that for an open set $\Om\subset X$, an arbitrary set $A\subset \Om$,
and any $\mu$-measurable sets
$E_1,E_2\subset X$, we have
\begin{equation}\label{eq:lattice property of sets of finite perimeter}
P(E_1\cap E_2,A)+P(E_1\cup E_2,A)\le P(E_1,A)+P(E_2,A);
\end{equation}
for a proof in the case $A=\Om$ see \cite[Proposition 4.7]{M},
and then the general case follows by approximation.
Using this fact as well as the lower semicontinuity of the total
variation with respect to $L_{\loc}^1$-convergence in open sets, we have
for any $E_1,E_2\ldots \subset X$ that
\begin{equation}\label{eq:perimeter of countable union}
P\Bigg(\bigcup_{j=1}^{\infty}E_j,\Om\Bigg)
\le \sum_{j=1}^{\infty}P(E_j,\Om).
\end{equation}

Applying the Poincar\'e inequality to sequences of approximating
$N^{1,1}_{\loc}$-functions in the definition of the total variation, we get
the following $\BV$ version:
for every ball $B(x,r)$ and every 
$u\in L^1_{\loc}(X)$, we have
\[
\int_{B(x,r)}|u-u_{B(x,r)}|\,d\mu
\le C_P r \Vert Du\Vert (B(x, r)).
\]
Recall here and from now on that we take the constant $\lambda$ to be $1$,
and so it does not appear in the inequalities.
For a $\mu$-measurable set $E\subset X$, by considering the two cases
$(\ch_E)_{B(x,r)}\le 1/2$ and $(\ch_E)_{B(x,r)}\ge 1/2$, from the above we get
the relative isoperimetric inequality
\begin{equation}\label{eq:relative isoperimetric inequality}
\min\{\mu(B(x,r)\cap E),\,\mu(B(x,r)\setminus E)\}\le 2 C_P rP(E,B(x,r)).
\end{equation}
From the $(1,1)$-Poincar\'e inequality,
by \cite[Theorem 4.21, Theorem 5.51]{BB}
we also get the following Sobolev inequality:
if $x\in X$, $0<r<\frac{1}{4}\diam X$, and $u\in N^{1,1}(X)$ with $u=0$
in $X\setminus B(x,r)$, then
\begin{equation}\label{eq:sobolev inequality}
\int_{B(x,r)} |u|\,d\mu \le C_S r \int_{B(x,r)}  g_u\,d\mu
\end{equation}
for a constant $C_S=C_S(C_d,C_P)\ge 1$.
For any $\mu$-measurable set $E\subset B(x,r)$, applying the Sobolev inequality
to a suitable sequence approximating $u$,
we get the isoperimetric inequality
\begin{equation}\label{eq:isop inequality with zero boundary values}
\mu(E)\le C_S r P(E,X).
\end{equation}

The lower and upper approximate limits of a function $u$ on an open set
$\Om$
are defined respectively by
\begin{equation}\label{eq:lower approximate limit}
u^{\wedge}(x):
=\sup\left\{t\in\R:\,\lim_{r\to 0}\frac{\mu(B(x,r)\cap\{u<t\})}{\mu(B(x,r))}=0\right\}
\end{equation}
and
\begin{equation}\label{eq:upper approximate limit}
u^{\vee}(x):
=\inf\left\{t\in\R:\,\lim_{r\to 0}\frac{\mu(B(x,r)\cap\{u>t\})}{\mu(B(x,r))}=0\right\}
\end{equation}
for $x\in \Om$.
Unlike Newton-Sobolev functions, we understand $\BV$ functions to be
equivalence classes of a.e. defined functions,
but $u^{\wedge}$ and $u^{\vee}$ are pointwise defined.

The $\BV$-capacity of a set $A\subset X$ is defined by
\[
\capa_{\BV}(A):=\inf \left(\Vert u\Vert_{L^1(X)}+\Vert Du\Vert(X)\right),
\]
where the infimum is taken over all $u\in\BV(X)$ with $u\ge 1$ in a neighborhood of $A$.
By \cite[Theorem 4.3]{HaKi} we know that for some constant
$C_{\textrm{cap}}=C_{\textrm{cap}}(C_d,C_P)\ge 1$ and every
$A\subset X$, we have
\begin{equation}\label{eq:Newtonian and BV capacities are comparable}
\capa_1(A)\le C_{\textrm{cap}}\capa_{\BV}(A).
\end{equation}
We also define a variational $\BV$-capacity for any $A\subset\Om$, with
$\Om\subset X$ open, by
\[
\rcapa^{\vee}_{\BV}(A,\Om):=\inf \Vert Du\Vert(X),
\]
where the infimum is taken over functions $u\in \BV(X)$ such that
$u^{\wedge}=u^{\vee}= 0$ $\mathcal H$-a.e. in $X\setminus \Om$ and
$u^{\vee}\ge 1$ $\mathcal H$-a.e. in $A$.
By \cite[Theorem 5.7]{L-SS} we know that
\begin{equation}\label{eq:variational one and BV capacity}
\rcapa_{1}(A,\Om)\le C_{\textrm{r}}\rcapa^{\vee}_{\BV}(A,\Om)
\end{equation}
for a constant $C_{\textrm{r}}=C_{\textrm{r}}(C_d,C_P)\ge 1$.

\textbf{Standing assumptions:} In Section \ref{sec:strong boundary points}
we will consider a different metric space $Z$ (which will later be taken
to be a subset of $X$), but in Sections \ref{sec:components} to \ref{sec:proof of the main result}
we will assume that $(X,d,\mu)$ is a complete, geodesic metric space that
is equipped with the doubling Radon measure $\mu$ and supports a
$(1,1)$-Poincar\'e inequality with $\lambda=1$.

\section{Strong boundary points}\label{sec:strong boundary points}

In this section we consider a complete metric space
$(Z,\widehat{d},\mu)$ where $\mu$ is a Borel regular outer measure and
doubling with constant $\widehat{C}_d\ge 1$. We define
the Mazurkiewicz metric
\begin{equation}\label{eq:widehat d c}
\widehat{d}_M(x,y):=\inf\{\diam F:\,F\subset Z\textrm{ is a continuum containing }x,y\},
\quad x,y\in Z,
\end{equation}
and we assume the space to be ``geodesic'' in the sense that
$\widehat{d}_M = \widehat{d}$.
As usual, a continuum means a compact connected set.

\begin{definition}
We say that $(x_0,\ldots,x_m)$ is an $\eps$-chain from $x_0$ to $x_m$
if $\widehat{d}(x_j,x_{j+1})<\eps$ for all $j=0,\ldots,m-1$.
\end{definition}

The following proposition gives the existence of a strong boundary point.

\begin{proposition}\label{prop:strong boundary point}
Let $x_0\in Z$, $R>0$, and let $E\subset Z$ be a $\mu$-measurable set
such that
\begin{equation}\label{eq:half measure assumption}
\frac{1}{2\widehat{C}_d^2}\le \frac{\mu(B(x_0,R)\cap E)}{\mu(B(x_0,R))}\le 1-\frac{1}{2\widehat{C}_d^2}.
\end{equation}
Then there exists a point $x\in B(x_0,6 R)$ such that
\begin{equation}\label{eq:desired density point}
\frac{1}{4 \widehat{C}_d^{12}}\le \liminf_{r\to 0}\frac{\mu(B(x,r)\cap E)}{\mu(B(x,r))}
 \le \limsup_{r\to 0}\frac{\mu(B(x,r)\cap E)}{\mu(B(x,r))}
\le 1-\frac{1}{4 \widehat{C}_d^{12}}.
\end{equation}
\end{proposition}

\begin{proof}
The proof is by suitable iteration, where we consider two options.

	\textbf{Case 1.}
Suppose that
\begin{equation}\label{eq:E smaller than half everywhere}
\frac{\mu(B(x,2^{-2}R)\cap E)}{\mu(B(x,2^{-2}R))}<\frac{1}{2}
\end{equation}
for all $x\in B(x_0,R)$; the case ``$>$'' is considered analogously.
Define a  ``bad'' set
\[
P:=\left\{x\in B(x_0,R):\,\frac{\mu(B(x,2^{-2j}R)\cap E)}{\mu(B(x,2^{-2j}R))}
\le \frac{1}{4\widehat{C}_d^6}\ \ \textrm{for some }j\in\N\right\}.
\]
For every $x\in P$ there is a radius $r_x\le R/20\le R$ such that
\[
\frac{\mu(B(x,5r_x)\cap E)}{\mu(B(x,5r_x))}\le \frac{1}{4\widehat{C}_d^6}.
\]
Thus $\{B(x,r_x)\}_{x\in P}$ is a covering of $P$.
By the $5$-covering theorem, pick a countable collection of
pairwise disjoint balls $\{B(x_j,r_j)\}_{j=1}^{\infty}$ such that
$P\subset \bigcup_{j=1}^{\infty}B(x_j,5r_j)$.
Now
\begin{align*}
\mu(P\cap E)\le \sum_{j=1}^{\infty}\mu(B(x_j,5r_j)\cap E)
&\le \frac{1}{4 \widehat{C}_d^6}\sum_{j=1}^{\infty}\mu(B(x_j,5r_j))\\
&\le \frac{1}{4 \widehat{C}_d^3}\sum_{j=1}^{\infty}\mu(B(x_j,r_j))\\
&\le \frac{1}{4 \widehat{C}_d^3}\mu(B(x_0,2R))\\
&\le \frac{1}{4 \widehat{C}_d^2}\mu(B(x_0,R)).
\end{align*}
Thus
\begin{align*}
\mu(P)
&=\mu(P\cap E)+\mu(P\setminus E)\\
&\le \frac{1}{4\widehat{C}_d^2}\mu(B(x_0,R))+\mu(B(x_0,R)\setminus E)\\
&\le \frac{1}{4\widehat{C}_d^2}\mu(B(x_0,R))+\Bigg(1-\frac{1}{2\widehat{C}_d^2}\Bigg)\mu(B(x_0,R))\quad\textrm{by }\eqref{eq:half measure assumption}\\
&\le \Bigg(1-\frac{1}{4\widehat{C}_d^2}\Bigg)\mu(B(x_0,R)).
\end{align*}
In particular,
there is a point $y\in B(x_0,R)\setminus P$.
Now there are two options.

	\textbf{Case 1(a).}
The first option is that for each $j\in\N$, we have
\[
\frac{\mu(B(y,2^{-2j}R)\cap E)}{\mu(B(y,2^{-2j}R))}<\frac{1}{2}
\]
and then in fact
\[
\frac{1}{4\widehat{C}_d^6}\le \frac{\mu(B(y,2^{-2j}R)\cap E)}{\mu(B(y,2^{-2j}R))}<\frac{1}{2},
\]
for all $j\in\N$,
since $y\in B(x_0,R)\setminus P$.
From this we easily find that \eqref{eq:desired density point} holds  (with $x=y$).

	\textbf{Case 1(b).}
The second option is that there is a smallest index $l\ge 2$ such that
\[
\frac{\mu(B(y,2^{-2l}R)\cap E)}{\mu(B(y,2^{-2l}R))}\ge\frac{1}{2}.
\]
Then
\[
\frac{1}{2\widehat{C}_d^2}\le \frac{\mu(B(y,2^{-2l+2}R)\cap E)}{\mu(B(y,2^{-2l+2}R))}
< \frac{1}{2},
\]
and also
\[
\frac{1}{4\widehat{C}_d^6}\le\frac{\mu(B(y,2^{-2j}R)\cap E)}{\mu(B(y,2^{-2j}R))}
<\frac{1}{2}\quad\textrm{for all }j=1,\ldots,l-2.
\]
Note that regardless of the direction of the inequality in
\eqref{eq:E smaller than half everywhere}, we get
\[
\frac{1}{2\widehat{C}_d^2}\le \frac{\mu(B(y,2^{-2l+2}R)\cap E)}{\mu(B(y,2^{-2l+2}R))}
< 1-\frac{1}{2\widehat{C}_d^2}
\]
and
\begin{equation}\label{eq:doubling constant six estimate}
\frac{1}{4\widehat{C}_d^6}\le\frac{\mu(B(y,2^{-2j}R)\cap E)}{\mu(B(y,2^{-2j}R))}
\le1-\frac{1}{4\widehat{C}_d^6}\quad\textrm{for all }j=1,\ldots,l-2.
\end{equation}

	\textbf{Case 2.}
Alternatively, suppose that we find two points $x,y\in B(x_0,R)$ such that
\[
\frac{\mu(B(x,2^{-2}R)\cap E)}{\mu(B(x,2^{-2}R))}\ge \frac{1}{2}
\]
and
\[
\frac{\mu(B(y,2^{-2}R)\cap E)}{\mu(B(y,2^{-2}R))}\le \frac{1}{2}.
\]
Then, using the fact that $\widehat{d}_M=\widehat{d}$,
we find a continuum $F$ that contains
$x$ and $y$ and is contained in $B(x_0,3 R)$.
Since $F$ is connected, for every $\eps>0$ there is an
$\eps$-chain in $F$ from $x$ to $y$. In particular, we
find an $R/4$-chain in $F$ from $x$ to $y$.
Let $z$ be the last point in the chain for which we have
\[
\frac{\mu(B(z,2^{-2}R)\cap E)}{\mu(B(z,2^{-2}R))}\ge \frac{1}{2}.
\]
If $z=y$, then we have
\[
\frac{\mu(B(z,2^{-2}R)\cap E)}{\mu(B(z,2^{-2}R))}= \frac{1}{2}.
\]
Else there exists $w\in F$ with $\widehat{d}(z,w)<R/4$ and
\[
\frac{\mu(B(w,2^{-2}R)\cap E)}{\mu(B(w,2^{-2}R))}< \frac{1}{2}\quad\textrm{and thus}
\quad \frac{\mu(B(w,2^{-2}R)\setminus E)}{\mu(B(z,2^{-1}R))}\ge 
\frac{1}{2\widehat{C}_d^2}.
\]
Now
\begin{align*}
\frac{\mu(B(z,2^{-1}R)\cap E)}{\mu(B(z,2^{-1}R))}
&= \frac{\mu(B(z,2^{-1}R))-\mu(B(z,2^{-1}R)\setminus E)}{\mu(B(z,2^{-1}R))}\\
&\le \frac{\mu(B(z,2^{-1}R))-\mu(B(w,2^{-2}R)\setminus E)}{\mu(B(z,2^{-1}R))}\\
&\le 1-\frac{1}{2\widehat{C}_d^2}.
\end{align*}
Conversely,
\[
\frac{\mu(B(z,2^{-1}R)\cap E)}{\mu(B(z,2^{-1}R))}
\ge\frac{\mu(B(z,2^{-2}R)\cap E)}{\widehat{C}_d\mu(B(z,2^{-2}R))}
\ge \frac{1}{2\widehat{C}_d}.
\]
In conclusion, there is $z\in B(x_0,3R)$ with
\[
\frac{1}{2\widehat{C}_d^2}\le \frac{\mu(B(z,2^{-1}R)\cap E)}{\mu(B(z,2^{-1}R))}\le 1-\frac{1}{2\widehat{C}_d^2};
\]
note that this holds also in the case $z=y$.

To summarize, in Case 1(a) we obtain infinitely many balls (and then we are done),
in Case 1(b) we obtain the $l-1$ new balls
$B(y,2^{-2}R),\ldots,B(y,2^{-2l+2}R)$, where $B(y,2^{-2l+2}R)$ satisfies
\eqref{eq:half measure assumption}, and in Case (2) we obtain one new
ball satisfying \eqref{eq:half measure assumption}.

By iterating the procedure
and concatenating the new balls obtained in each step to the previous
list of balls, we find a sequence of balls with center points
$x_k\in B(x_{k-1},3 r_{k-1})$ and radii
$r_k$ such that $r_0=R$, $r_k\in [r_{k-1}/4,r_{k-1}/2]$, and
(recall \eqref{eq:doubling constant six estimate})
\[
\frac{1}{4\widehat{C}_d^6}\le \frac{\mu(B(x_k,r_k)\cap E)}{\mu(B(x_k,r_k))}
\le 1-\frac{1}{4\widehat{C}_d^6}
\]
for all $k\in\N$.
(Note that several consecutive balls in this sequence will have the same center
points if they are obtained from Case 1.)
By completeness of the space
we find $x\in Z$ such that
$x_k\to x$. For each $l=0,1,\ldots$ we have
\[
d(x,x_l)\le  \sum_{k=l}^{\infty}d(x_k,x_{k+1})
\le 3\sum_{k=l}^{\infty}r_k\le 6 r_l.
\]
In particular,
$d(x,x_0)\le 6 R$.
Now $B(x_l,r_l)\subset B(x,7 r_l)\subset B(x_l,13  r_l)$
for all $l\in\N$, and so
\[
\frac{\mu(B(x,7  r_l)\cap E)}{\mu(B(x,7  r_l))}
\ge \frac{\mu(B(x_l,r_l)\cap E)}{\mu(B(x_l,13 r_l))}
\ge \frac{1}{\widehat{C}_d^{4}}\frac{\mu(B(x_l,r_l)\cap E)}{\mu(B(x_l,r_l))}
\ge \frac{1}{4 \widehat{C}_d^{10}}
\]
and similarly
\[
\frac{\mu(B(x,7  r_l)\setminus E)}{\mu(B(x,7  r_l))}
\ge \frac{\mu(B(x_l,r_l)\setminus E)}{\mu(B(x_l,13  r_l))}
\ge \frac{1}{\widehat{C}_d^{4}}
\frac{\mu(B(x_l,r_l)\setminus E)}{\mu(B(x_l,r_l))}
\ge \frac{1}{4 \widehat{C}_d^{10}}.
\]
It follows that
\[
\liminf_{r\to 0}\frac{\mu(B(x,r)\cap E)}{\mu(B(x,r))}
\ge \frac{1}{4 \widehat{C}_d^{12}}
\]
and
\[
\liminf_{r\to 0}\frac{\mu(B(x,r)\setminus E)}{\mu(B(x,r))}
\ge \frac{1}{4 \widehat{C}_d^{12}},
\]
proving \eqref{eq:desired density point}.
\end{proof}

\begin{corollary}\label{cor:density points}
Let $x_0\in Z$, $R>0$, and let $E\subset Z$ be a $\mu$-measurable set
such that
\[
0< \mu(B(x_0,R)\cap E)<\mu(B(x_0,R)).
\]
Then there exists a point $x\in B(x_0,9 R)$ such that
\begin{equation}\label{eq:strong boundary point}
\frac{1}{4 \widehat{C}_d^{12}}\le \liminf_{r\to 0}\frac{\mu(B(x,r)\cap E)}{\mu(B(x,r))}
\le \limsup_{r\to 0}\frac{\mu(B(x,r)\cap E)}{\mu(B(x,r))}
\le 1-\frac{1}{4 \widehat{C}_d^{12}}.
\end{equation}
\end{corollary}
\begin{proof}
Again consider two cases. The first is that
we find two points $y,z\in B(x_0,R)$ such that
\[
\frac{\mu(B(y,2^{-1}R)\cap E)}{\mu(B(y,2^{-1}R))}\ge \frac{1}{2}
\quad\textrm{and}\quad
\frac{\mu(B(z,2^{-1}R)\cap E)}{\mu(B(z,2^{-1}R))}\le \frac{1}{2}.
\]
Then just as in the proof of Proposition \ref{prop:strong boundary point}
Case 2, we find $w\in B(x_0,3R)$ with
\[
\frac{1}{2\widehat{C}_d^2}\le \frac{\mu(B(w,R)\cap E)}{\mu(B(w,R))}
\le 1-\frac{1}{2\widehat{C}_d^2}.
\]
Now Proposition \ref{prop:strong boundary point} gives a point $x\in B(w,6R)\subset B(x_0,9R)$ such that \eqref{eq:strong boundary point} holds.

The second possible case is that for all $y\in B(x_0,R)$ we have
\[
\frac{\mu(B(y,2^{-1}R)\cap E)}{\mu(B(y,2^{-1}R))}< \frac{1}{2}
\]
(the case ``$>$'' being analogous).
By Lebesgue's differentiation theorem, we find a point
$y\in I_E\cap B(x_0,R)$ (recall \eqref{eq:measure theoretic interior}) and then
it is easy to find a radius $0<r\le R/2$ such that
\[
\frac{1}{2\widehat{C}_d}\le \frac{\mu(B(y,r)\cap E)}{\mu(B(y,r))}
<\frac{1}{2}.
\]
Now Proposition \ref{prop:strong boundary point} again gives a point
$x\in B(y,6r)\subset B(x_0,4R)$ such that
\eqref{eq:strong boundary point} holds.
\end{proof}

\section{Components of sets of finite perimeter}\label{sec:components}

In Sections \ref{sec:components} to \ref{sec:proof of the main result}
we assume that $(X,d,\mu)$ is a complete, geodesic metric space that
is equipped with the doubling measure $\mu$ and supports a
$(1,1)$-Poincar\'e inequality.

In this section we consider connected components,
or components for short, of sets of finite perimeter.
The following is the main result of the section.

\begin{proposition}\label{prop:connected components}
Let $B(x,R)$ be a ball with $0<R<\frac{1}{4}\diam X$ and let
$F\subset X$ be a closed set with $P(F,X)<\infty$.
Denote the components of $F\cap \overline{B}(x,R)$ having nonzero
$\mu$-measure by $F_1,F_2,\ldots$. Then $\mu\left(\overline{B}(x,R)\cap F\setminus \bigcup_{j=1}^{\infty}F_j\right)=0$,
$P(F_j,B(x,R))<\infty$ for all $j\in\N$,
and for any sets $A_j\subset F_j$ with $P(A_j,B(x,R))<\infty$ for
all $j\in\N$ we have
\[
P\Bigg(\bigcup_{j=1}^{\infty}A_j,B(x,R)\Bigg)=\sum_{j=1}^{\infty}P(A_j,B(x,R)).
\]
\end{proposition}

Of course, there may be only finitely many $F_j$'s, and so we will always understand
that some $F_j$'s can be empty. In fact, supposing that $\mu(F\cap B(x,R))>0$,
we will know only after Lemma \ref{lem:H has measure zero} that
any $F_j$'s are nonempty.

Next we gather a number of preliminary results.
Recall the definition of $1$-quasiopen sets from page
\pageref{quasiopen}.

\begin{proposition}[{\cite[Proposition 4.2]{L-Fed}}]\label{prop:set of finite perimeter is quasiopen}
	Let $\Omega\subset X$ be open and let $F\subset X$ be $\mu$-measurable with
	$P(F,\Omega)<\infty$. Then the sets $I_F\cap\Omega$ and $O_F\cap\Omega$ are $1$-quasiopen.
\end{proposition}

\begin{proposition}\label{prop:ae curve goes through boundary}
Let
$F\subset X$ with $P(F,X)<\infty$.
Then for $1$-a.e. curve $\gamma$, $\gamma^{-1}(I_F)$
and $\gamma^{-1}(O_F)$ are relatively open subsets of $[0,\ell_{\gamma}]$.
\end{proposition}
\begin{proof}
	By Proposition \ref{prop:set of finite perimeter is quasiopen},
	the sets $I_F$ and $O_F$ are $1$-quasiopen. Then
by \cite[Remark 3.5]{S2}, they
are also \emph{$1$-path open},
meaning that for $1$-a.e. curve $\gamma$ in $X$,
the sets $\gamma^{-1}(I_F)$ and $\gamma^{-1}(O_F)$
are relatively open subsets of $[0,\ell_{\gamma}]$.
\end{proof}

For any set $A\subset X$, we define the \emph{measure-theoretic closure} as
\begin{equation}\label{eq:measure theoretic closure}
\overline{A}^m:=I_A\cup \partial^*A.
\end{equation}

\begin{lemma}\label{lem:inner capacity}
Let $B(x,R)$ be a ball with $0<R<\frac{1}{4}\diam X$ and let $E_1\supset E_2\supset \ldots$ such that
$P(E_j,B(x,R))<\infty$ for all $j\in\N$,
and $\mu(E_j)\to 0$ and $P(E_j,B(x,R))\to 0$ as $j\to\infty$.
Let $0<r<R$. Then
\[
\capa_1(\overline{E_j}^m\cap B(x,r))\to 0.
\]
\end{lemma}
\begin{proof}
Take a cutoff function $\eta\in \Lip_c(B(x,R))$ with $0\le \eta\le 1$ on $X$,
$\eta=1$ in $B(x,r)$, and $g_\eta\le 2/(R-r)$, where $g_{\eta}$ is the minimal
$1$-weak upper gradient of $\eta$.
Then for all $j\in\N$, by a Leibniz rule (see
\cite[Proposition 4.2]{KKST3}) we have
\[
\Vert D(\ch_{E_j}\eta)\Vert(X)=\Vert D(\ch_{E_j}\eta)\Vert(B(x,R))
\le \frac{2\mu(E_j)}{R-r}+P(E_j,B(x,R))\to 0
\]
as $j\to\infty$.
By \eqref{eq:variational one and BV capacity} and the fact that
$(\ch_{E_j}\eta)^{\vee}=1$ in $\overline{E_j}^m\cap B(x,r)$, we get
\begin{align*}
\rcapa_1(\overline{E_j}^m\cap B(x,r),B(x,R))
&\le C_{\textrm{r}}\rcapa_{\BV}^{\vee}(\overline{E_j}^m\cap B(x,r),B(x,R))\\
&\le \Vert D(\ch_{E_j}\eta)\Vert(X)\to 0\quad\textrm{as }j\to\infty.
\end{align*}
Then by the Sobolev inequality \eqref{eq:sobolev inequality} we easily
get
\[
\capa_1(\overline{E_j}^m\cap B(x,r))\to 0.
\]
\end{proof}

The variation measure is always absolutely continuous with respect to
the $1$-capacity, in the following sense.

\begin{lemma}[{\cite[Lemma 3.8]{L-SA}}]\label{lem:variation measure and capacity}
	Let $\Omega\subset X$ be an open set and
	let $u\in L^1_{\loc}(\Omega)$ with $\Vert Du\Vert(\Omega)<\infty$. Then for every $\eps>0$ there exists $\delta>0$ such that if $A\subset \Omega$ with $\capa_1 (A)<\delta$, then $\Vert Du\Vert(A)<\eps$.
\end{lemma}

\begin{lemma}\label{lem:coincidence of perimeter}
Let $\Om\subset X$ be open,
let $F_1\subset F_2\subset X$ with $P(F_1,\Om)<\infty$ and $P(F_2,\Om)<\infty$,
and let $A\subset \Om$
such that for all $x\in A$, we have
\[
\lim_{r\to 0}\frac{\mu(B(x,r)\cap (F_2\setminus F_1))}{\mu(B(x,r))}=0.
\]
Then $P(F_1,A)=P(F_2,A)$.
\end{lemma}
\begin{proof}
First note that $P(F_2\setminus F_1,\Om)<\infty$
by \eqref{eq:lattice property of sets of finite perimeter},
and then by \eqref{eq:def of theta} we have
\[
P(F_2\setminus F_1,A)=0.
\]
Using \eqref{eq:lattice property of sets of finite perimeter} again, we have
\[
P(F_2,A)\le P(F_1,A)+P(F_2\setminus F_1,A)=P(F_1,A)
\]
and
\[
P(F_1,A)\le P(F_2,A)+P(F_2\setminus F_1,A)=P(F_2,A).
\]
\end{proof}

The following lemma says that perimeter can always be controlled by
the measure of a suitable ``curve boundary''.

\begin{lemma}\label{lem:perimeter controlled by boundary}
Let $\Om\subset X$ be open, let $E\subset X$ be closed,
and let $A\subset \Om$ be such that $1$-a.e. curve
$\gamma$ in $\Om$ with $\gamma(0)\in I_E$ and
$\gamma(\ell_{\gamma})\in X\setminus E$
intersects $A$. Then $P(E,\Om)\le C_d\mathcal H(A)$. 
\end{lemma}

\begin{proof}
We can assume that $\mathcal H(A)<\infty$.
Fix $\eps>0$. We find a covering of $A$
by balls $\{B_j=B(x_j,r_j)\}_{j\in I}$, with $I\subset\N$, such that $r_j\le \eps$ 
and 
\begin{equation}\label{eq:covering for A}
\sum_{j\in I}\frac{\mu(B_j)}{r_j}\le \mathcal{H}(A)+\eps.
\end{equation}
Denote the exceptional family of curves by $\Gamma$.
Take a nonnegative Borel function $\rho$ such that $\Vert \rho\Vert_{L^1(\Om)}<\eps$
and $\int_{\gamma}\rho\,ds\ge 1$ for all $\gamma\in\Gamma$.
Let
\[
g:=\sum_{j\in I}\frac{\ch_{2B_j}}{r_j}+\rho.
\]
Then let
\[
u(x):=\min\left\{1,\inf \int_{\gamma}g\,ds\right\},
\]
where the infimum is taken over curves $\gamma$ (also constant curves)
in $\Om$ with $\gamma(0)= x$ and
$\gamma(\ell_{\gamma})\in \Om\setminus \left(E\cup\bigcup_{j\in I}2B_j\right)$.
We know that $g$ is an upper gradient of $u$ in $\Om$,
see \cite[Lemma 5.25]{BB}. Moreover, $u$ is $\mu$-measurable
by \cite[Theorem 1.11]{JJRRS}; strictly speaking this result is written for
functions defined on the whole space, but the proof clearly works also for functions
defined in an open set such as $\Om$.
If $x\in \Om\setminus \left(E\cup\bigcup_{j\in I}2B_j\right)$,
clearly $u(x)=0$.
If $x\in I_E\setminus \bigcup_{j\in I}2B_j$, consider any curve
$\gamma$ in $\Om$ with $\gamma(0)= x$ and
$\gamma(\ell_{\gamma})\in \Om\setminus \left(E\cup\bigcup_{j\in I}2B_j\right)$.
Then either $\int_{\gamma}\rho\,ds\ge 1$ or there is $t$ such that
$\gamma(t)\in A$. In the latter case,
for some $j\in I$ we have $\gamma(t)\in B_j$. Then
\[
\int_{\gamma}g\,ds\ge \int_{\gamma}\frac{\ch_{2B_j}}{r_j}\,ds\ge 1.
\]
Thus  $u(x)=1$, and so by Lebesgue's differentiation theorem we have
$u=\ch_E$ a.e. in $\Om\setminus \bigcup_{j\in I}2B_j$. Thus
\begin{align*}
\int_{\Omega}|u-\ch_E|\, d\mu
&\le \int_{\Omega} \ch_{\bigcup_{j\in I}2B_j}\, d\mu\le \sum_{j\in I}\mu(2B_j)
\le \eps\sum_{j\in I} \frac{\mu(2B_j)}{r_j}
\le \eps (C_d\mathcal{H}(A)+\eps).
\end{align*}
Moreover, using \eqref{eq:covering for A} we get
\[
\int_{\Om}g\,d\mu\le  \sum_{j\in I}\int_{\Om}\frac{\ch_{2B_j}}{r_j}\,d\mu
+\int_{\Om}\rho\,d\mu\le C_d\mathcal H(A)+C_d \eps +\eps.
\]

Now for each $i\in\N$, use the above construction to obtain functions
$u_i\in N^{1,1}_{\loc}(\Omega)$ and upper gradients
$g_i\in L^1(\Omega)$ corresponding to $\eps=1/i$.
We have
\[
\int_{\Omega}|u_i-\ch_E|\, d\mu\le i^{-1} (C_d\mathcal{H}(A)+i^{-1})\to 0
\quad\textrm{as }i\to \infty
\]
and thus
\[
P(E,\Om)\le \liminf_{i\to\infty}\int_{\Om}g_i\,d\mu
\le \liminf_{i\to\infty}(C_d\mathcal H(A)+C_d i^{-1}+i^{-1})=C_d\mathcal H(A).
\]

\end{proof}

\begin{proposition}\label{prop:sum of perimeters of components}
Let $B(x,R)$ be a ball with $0<R<\frac{1}{4}\diam X$ and let $F\subset X$ be a closed set
with $P(F,X)<\infty$.
Denote the components of $F\cap \overline{B}(x,R)$ having nonzero
$\mu$-measure by $F_1,F_2,\ldots$.
Then
\[
\sum_{j=1}^{\infty}P(F_j,B(x,R))<\infty,
\]
and for any sets $A_j\subset F_j$
with $P(A_j,B(x,R))<\infty$ for all $j\in\N$ we have
\begin{equation}\label{eq:perimeter of union and sum of Ajs is the same}
P\Bigg(\bigcup_{j=1}^{\infty}A_j,B(x,R)\Bigg)
=\sum_{j=1}^{\infty}P(A_j,B(x,R)).
\end{equation}
\end{proposition}

\begin{proof}
Let $\Gamma_b$ be the exceptional family of curves of
Proposition \ref{prop:ae curve goes through boundary};
then $\Mod_1(\Gamma_b)=0$.
Consider a component $F_j$; it is a closed set.
Consider a curve $\gamma\notin \Gamma_b$ in $B(x,R)$ with $\gamma(0)\in I_{F_j}$ and
$\gamma(\ell_{\gamma})\in X\setminus F_j$. Then $\gamma(0)\in I_F$.
Take
\[
t:=\max\{s\in [0,\ell_{\gamma}]:\,\gamma([0,s])\subset F_j\}.
\]
Clearly $t<\ell_{\gamma}$.
There cannot exist $\delta>0$ such that
$\gamma(s)\in F$ for all $s\in (t,t+\delta)$
because this would connect $F_j$ with at least
one other component of $F\cap \overline{B}(x,R)$.
Thus there are points $s_j\searrow t$ with $\gamma(s_j)\in X\setminus F\subset O_F$.
By Proposition \ref{prop:ae curve goes through boundary},
this implies that either $\gamma(t)\in\partial^*F$ or $\gamma(t)\in O_{F}$.
In the latter case, there is a point $\widetilde{t}\in (0,t)$
with $\gamma(\widetilde{t})\in\partial^*F$.
In both cases, we have found $t$ such that $\gamma(t)\in \partial^*F\cap F_j$.
Thus by Lemma \ref{lem:perimeter controlled by boundary},
\[
P(F_j,B(x,R))\le C_d\mathcal H(\partial^*F\cap F_j)
\]
and so
\begin{equation}\label{eq:perimeter sum is finite}
\begin{split}
\sum_{j=1}^{\infty}P(F_j,B(x,R))
&\le C_d \sum_{j=1}^{\infty} \mathcal H(\partial^*F\cap F_j)\\
&\le C_d \mathcal H(\partial^*F)\\
&\le C_d\alpha^{-1}P(F,X)\quad\textrm{by }\eqref{eq:def of theta}\\
&<\infty,
\end{split}
\end{equation}
as desired. Next note that one inequality in \eqref{eq:perimeter of union and sum of Ajs is the same} follows from \eqref{eq:perimeter of countable union}. To prove the other one, note that the sets $F_j$ are closed and then in fact compact,
and so for any $\mu$-measurable
sets $A_j\subset F_j$  with $P(A_j,B(x,R))<\infty$ for all $j\in\N$, we have
\begin{equation}\label{eq:distance between Aj and Ak}
\dist(A_j,A_k)\ge \dist(F_j,F_k)>0
\end{equation}
for all $j\neq k$. Take $N,M\in\N$ with $N\le M$. We have
(recall \eqref{eq:measure theoretic closure})
\begin{equation}\label{eq:perimeter of union of Ajs}
\begin{split}
P\Bigg(\bigcup_{j=1}^{\infty}A_j,B(x,R)\Bigg)
&\ge P\Bigg(\bigcup_{j=1}^{\infty}A_j,B(x,R)\setminus \overline{\bigcup_{j=M+1}^{\infty}A_j}^m\Bigg)\\
&=P\Bigg(\bigcup_{j=1}^{M}A_j,B(x,R)\setminus 
\overline{\bigcup_{j=M+1}^{\infty}A_j}^m\Bigg)\quad\textrm{by Lemma }\ref{lem:coincidence of perimeter}\\
&=\sum_{j=1}^{M} P\Bigg(A_j,B(x,R)\setminus\overline{\bigcup_{j=M+1}^{\infty}A_j}^m\Bigg)\quad\textrm{by }\eqref{eq:distance between Aj and Ak}\\
&\ge \sum_{j=1}^{N} P\Bigg(A_j,B(x,R)\setminus\overline{\bigcup_{j=M+1}^{\infty}A_j}^m\Bigg).
\end{split}
\end{equation}
By \eqref{eq:perimeter of countable union} and \eqref{eq:perimeter sum is finite},
we have
\[
P\Bigg(\bigcup_{j=M+1}^{\infty}F_j,B(x,R)\Bigg)
\le 
\sum_{j=M+1}^{\infty}P(F_j,B(x,R))
\to 0\quad\textrm{as }M\to \infty.
\]
Then by Lemma \ref{lem:inner capacity} we have
 \[
 \capa_1\Bigg(\overline{\bigcup_{j=M+1}^{\infty}A_j}^m\cap B(x,r)\Bigg)
 \le \capa_1\Bigg(\overline{\bigcup_{j=M+1}^{\infty}F_j}^m\cap B(x,r)\Bigg)
 \to 0\quad\textrm{as }M\to \infty
 \]
 for all $0<r<R$.
From \eqref{eq:perimeter of union of Ajs} and
Lemma \ref{lem:variation measure and capacity} we now get
\[
P\Bigg(\bigcup_{j=1}^{\infty}A_j,B(x,R)\Bigg)\ge \sum_{j=1}^{N} P(A_j,B(x,r)).
\]
Letting $r\nearrow R$ and $N\to\infty$, we get the conclusion.
\end{proof}

For any nonnegative $g\in L^1_{\loc}(X)$, define the centered
Hardy-Littlewood maximal function
\[
\mathcal M g(x):=\sup_{r>0}\,\vint{B(x,r)}g\,d\mu,\quad x\in X.
\]

Recall the definition of the exponent $s>1$ from
\eqref{eq:homogenous dimension}.
The argument in the following lemma was inspired by the study
of the so-called $\textrm{MEC}_p$-property in \cite{JJRRS}.

\begin{lemma}\label{lem:finding a positive measure component}
Let $B(x_0,r)$ be a ball and let $V\subset X$ be an open set with
\[
\capa_1(V\cap B(x_0,r))< \frac{1}{20 \cdot 10^s C_P C_d^7}\frac{\mu(B(x_0,r))}{r}.
\]
Then there is a connected subset of $\overline{B}(x_0,r/2)\setminus V$
with measure at least $\mu(B(x_0,r))/(4\cdot 10^s C_d^2)$.
\end{lemma}

\begin{proof}
Take $u\in N^{1,1}(X)$  with $u=1$ in $V\cap B(x_0,r)$ and
\[
\Vert u\Vert_{N^{1,1}(X)}<\frac{1}{20 \cdot 10^s C_P C_d^7}\frac{\mu(B(x_0,r))}{r}.
\]
Thus there is an upper gradient $g$ of $u$ with
\[
\Vert g\Vert_{L^1(X)}<\frac{1}{20 \cdot 10^s C_P C_d^7}\frac{\mu(B(x_0,r))}{r}.
\]
By the Vitali-Carath\'eodory theorem
(see e.g. \cite[p. 108]{HKST15}) we can assume that $g$ is lower semicontinuous. 
We define
\[
A:=\{\mathcal M g> (10C_P C_d^2 r)^{-1}\}\quad\textrm{and}\quad
D:=\{u\ge 1/2\}.
\]
Then by the weak $L^1$-boundedness of the maximal function
(see e.g. \cite[Lemma 3.12]{BB}) as well as
\eqref{eq:homogenous dimension}, we estimate
\[
\mu(A)\le 10 C_P C_d^5 r\Vert g\Vert_{L^1(X)}\le \frac{1}{2 \cdot 10^s C_d^2}\mu(B(x_0,r))
\le \frac{1}{2}\mu(B(x_0,r/10)).
\]
Similarly,
\[
\mu(D)\le 2\Vert u\Vert_{L^1(X)}\le \frac{1}{4}\mu(B(x_0,r/10)),
\]
and then
\begin{equation}\label{eq:measure of complement of A D}
\mu(B(x_0,r/10)\setminus (A\cup D))\ge \frac{1}{4}\mu(B(x_0,r/10))
\ge \frac{\mu(B(x_0,r))}{4\cdot 10^s C_d^2}.
\end{equation}
In particular, we can fix $x\in B(x_0,r/10)\setminus (A\cup D)$.
Let $\delta:=(100C_P C_d^2 r)^{-1}$.
For every $k\in\N$, let $g_k:=\min\{g,k\}$ and
\[
v_k(y):=\inf\int_{\gamma}(g_k+\delta)\,ds,\quad y\in B(x_0,r/2),
\]
where the infimum is taken over curves $\gamma$ (also constant curves)
in $B(x_0,r/2)$ with $\gamma(0)= x$ and $\gamma(\ell_{\gamma})=y$.
Then $g_k+\delta\le g+\delta$ is an upper gradient of $v_k$ in
$B(x_0,r/2)$
(see \cite[Lemma 5.25]{BB}) and $v_k$ is $\mu$-measurable
by \cite[Theorem 1.11]{JJRRS}.
Since the space is geodesic, each $v_k$ is $(k+\delta)$-Lipschitz
in $B(x_0,r/10)$ and thus all points in $B(x_0,r/10)$
are Lebesgue points of $v_k$.
Define $B_j:=B(x,2^{-j+1}r/10)$, for $j=0,1\ldots$.
By the Poincar\'e inequality,
\begin{equation}\label{eq:telescope at x}
\begin{split}
|v_k(x)-(v_k)_{B_0}|
\le \sum_{j=0}^{\infty}|(v_k)_{B_{j+1}}-(v_k)_{B_{j}}|
&\le C_d\sum_{j=0}^{\infty}\, \vint{B_j}|v_k-(v_k)_{B_j}|\,d\mu\\
&\le C_d C_P \sum_{j=0}^{\infty}\frac{2^{-j+1}r}{10}\vint{B_j}(g+\delta)\,d\mu\\
&\le  C_d C_P r(\mathcal M g(x)+ \delta)\\
&\le 1/8.
\end{split}
\end{equation}
Similarly, for every
$y\in B(x_0,r/10)\setminus (A\cup D)$ we have
\begin{equation}\label{eq:telescope at y}
|v_k(y)-(v_k)_{B(y,r/5)}|\le 1/8
\end{equation}
and 
\begin{equation}\label{eq:middle term}
\begin{split}
|(v_k)_{B(x,r/5)}-(v_k)_{B(y,r/5)}|
&\le 2C_d^2\vint{B(x,2r/5)}|v_k-(v_k)_{B(x,2r/5)}|\,d\mu\\
&\le 2C_d^2 C_P r\vint{B(x,2r/5)}(g+\delta)\,d\mu\\
&\le 2C_d^2 C_P r(\mathcal Mg(x)+\delta)\\
&\le 1/4.
\end{split}
\end{equation}
Combining \eqref{eq:telescope at x}, \eqref{eq:telescope at y},
and \eqref{eq:middle term}, we get
\[
v_k(y)= |v_k(x)-v_k(y)|\le 1/2.
\]
This means that there is a curve $\gamma_k$ in $B(x_0,r/2)$ with
$\gamma_k(0)= x$, $\gamma_k(\ell_{\gamma_k})=y$, and
$\int_{\gamma_k}(g_k+\delta)\,ds\le 1/2$, for every $k\in\N$.
Note that
\[
\ell_{\gamma_k}\le \frac{1}{\delta}\int_{\gamma_k}(g_k+\delta)\,ds\le \frac{1}{2\delta}.
\]
Consider the reparametrizations
$\widetilde{\gamma}_k(t):=\gamma_k(t\ell_{\gamma_k})$, $t\in [0,1]$.
By the Arzela-Ascoli theorem (see e.g. \cite[p. 169]{Roy}),
passing to a subsequence (not relabeled)
we find $\widetilde{\gamma}\colon [0,1]\to X$ such that
$\widetilde{\gamma}_k\to \widetilde{\gamma}$ uniformly.
It is straightforward to check that $\widetilde{\gamma}$ is continuous
and rectifiable.
Let $\gamma$ be the parametrization of $\widetilde{\gamma}$ by arc-length;
then $\gamma(0)= x$ and $\gamma(\ell_{\gamma})=y$, and
by \cite[Lemma 2.2]{JJRRS}, we have for every $k_0\in\N$ that
\[
\int_{\gamma}g_{k_0}\,ds\le \liminf_{k\to\infty}\int_{\gamma_k}g_{k_0}\,ds
\le \liminf_{k\to\infty}\int_{\gamma_k}g_{k}\,ds\le 1/2.
\]
Letting $k_0\to\infty$, we obtain
\[
\int_{\gamma}g\,ds\le 1/2.
\]
Note that if $\gamma$ intersected a point $z\in V$, then we would have
\[
\int_{\gamma}g\,ds \ge |u(x)-u(z)|> |1/2-1|=1/2,
\]
so this is not possible. Thus $\gamma$ is in $\overline{B}(x_0,r/2)\setminus V$;
let us denote this curve, and also its image, by $\gamma_y$.
Define the desired connected set as the union
\[
\bigcup_{y\in B(x_0,r/10)\setminus (A\cup D)}\gamma_y.
\]
By \eqref{eq:measure of complement of A D} this has measure at least
$\mu(B(x_0,r))/(4\cdot 10^s C_d^2)$.
\end{proof}

\begin{lemma}\label{lem:H has measure zero}
	Let $B(x,R)$ be a ball with $0<R<\frac{1}{4}\diam X$
	and let $F\subset X$ be a closed set
with $P(F,X)<\infty$.
Denote the components of $F\cap \overline{B}(x,R)$ having nonzero
$\mu$-measure by $F_1,F_2,\ldots$, and
	$H:=\overline{B}(x,R)\cap F\setminus \bigcup_{j=1}^{\infty} F_j$.
	Then $\mu(H)=0$.
\end{lemma}
\begin{proof}
It follows from Proposition \ref{prop:sum of perimeters of components}
that $P\left(\bigcup_{j=1}^{\infty} F_j,B(x,R)\right)<\infty$, and then
by \eqref{eq:lattice property of sets of finite perimeter} also
$P(H,B(x,R))<\infty$.
By \eqref{eq:def of theta} and
a standard covering argument (see e.g. the proof of \cite[Lemma 2.6]{KKST3}),
we find that
\[
\lim_{r\to 0}r\frac{P\left(\bigcup_{j=1}^{\infty} F_j,B(y,r)\right)}{\mu(B(y,r))}=0
\]
for all $y\in B(x,R)\setminus \left(\partial^*\big(\bigcup_{j=1}^{\infty} F_j\big)\cup N\right)$, with $\mathcal H(N)=0$, in particular for all
$y\in B(x,R)\cap I_H\setminus N$.

Take $y\in B(x,R)\cap I_H\setminus N$ (if it exists).
We find arbitrarily small $r>0$ such that $B(y,r) \subset B(x,R)$ and
\begin{equation}\label{eq:complement of H small}
\frac{\mu(B(y,r)\setminus H)}{\mu(B(y,r))}\le \frac{1}{80 \cdot 10^s C_P C_d^8 C_{\textrm{cap}}}
\end{equation}
and
\[
r\frac{P\left(\bigcup_{j=1}^{\infty} F_j,B(y,r)\right)}{\mu(B(y,r))}
\le \frac{1}{80 \cdot 10^s C_P C_d^8 C_{\textrm{Cap}}}.
\]
Now suppose that
\[
P(H,B(y,r))\le \frac{1}{80 \cdot 10^s C_P C_d^8 C_{\textrm{cap}}}\frac{\mu(B(y,r))}{r}.
\]
Then since  $H\cup\bigcup_{j=1}^{\infty}F_j=F\cap \overline{B}(x,R)$,
by \eqref{eq:lattice property of sets of finite perimeter} we get
\begin{align*}
P(F,B(y,r))
&\le P(H,B(y,r))+P\Bigg(\bigcup_{j=1}^{\infty}F_j,B(y,r)\Bigg)\\
&\le \frac{1}{40 \cdot 10^s C_P C_d^8 C_{\textrm{cap}}}\frac{\mu(B(y,r))}{r}.
\end{align*}
Define the Lipschitz function
\[
\eta:=\max\left\{0,1-\frac{\dist(\cdot,B(y,r/2))}{r/2}\right\},
\]
so that $0\le \eta\le 1$ on $X$, $\eta=1$ in $B(y,r/2)$,
$\eta=0$ in $X\setminus B(y,r)$, and
$g_{\eta}\le (2/r)\ch_{B(y,r)}$ (see \cite[Corollary 2.21]{BB}).
Then by a Leibniz rule (see
\cite[Proposition 4.2]{KKST3}), we have
\begin{align*}
\capa_{\BV}(B(y,r/2)\setminus F)
&\le \Vert D(\eta\ch_{X\setminus F})\Vert(X)\\
&\le P(F,B(y,r))+2\frac{\mu(B(y,r)\setminus F)}{r}\\
&\le P(F,B(y,r))+2\frac{\mu(B(y,r)\setminus H)}{r}\\
&\le \frac{1}{20 \cdot 10^s C_P C_d^8 C_{\textrm{cap}}}\frac{\mu(B(y,r))}{r}.
\end{align*}
Then by \eqref{eq:Newtonian and BV capacities are comparable},
\[
\capa_{1}(B(y,r/2)\setminus F)\le 
\frac{1}{20 \cdot 10^s C_P C_d^8}\frac{\mu(B(y,r))}{r}
<\frac{1}{20 \cdot 10^s C_P C_d^7}\frac{\mu(B(y,r/2))}{r/2}.
\]
Then by Lemma \ref{lem:finding a positive measure component},
there is a connected subset of $F\cap \overline{B}(y,r/4)$
with measure at least
\[
\frac{\mu(B(y,r/2))}{4\cdot 10^s C_d^2}\ge \frac{\mu(B(y,r))}{4\cdot 10^s C_d^3}.
\]
By \eqref{eq:complement of H small} this must be (partially) contained in $H$,
a contradiction since $H$ contains no components of nonzero measure.
Thus for all $y\in I_H\cap B(x,R)\setminus N$, we have
\[
\limsup_{r\to 0}r\frac{P(H,B(y,r))}{\mu(B(y,r))}
\ge \frac{1}{80 \cdot 10^s C_P C_d^8 C_{\textrm{cap}}}.
\]
By a simple covering argument, it follows that
\[
\mu(I_H\cap B(x,R)\setminus N)\le \eps\cdot 80 \cdot 
10^s C_P C_d^{11} C_{\textrm{cap}} P(H,B(x,R))
\]
for every $\eps>0$. Thus $\mu(H\cap B(x,R)\setminus N)=0$
and so $\mu(H\cap B(x,R))=0$.
Since the space $X$ is geodesic, by \cite[Corollary 2.2]{Buc}
we know that $\mu(\{y\in X:\,d(y,x)=R\})=0$ and so
in fact $\mu(H)=0$.
\end{proof}

\begin{proof}[Proof of Proposition \ref{prop:connected components}]
	This follows from
Proposition \ref{prop:sum of perimeters of components} and
Lemma \ref{lem:H has measure zero}.
\end{proof}

\section{Functions of least gradient}\label{sec:least gradient}

In this section we consider functions of least gradient, or more precisely
superminimizers and solutions of obstacle problems
in the case $p=1$. We will follow the definitions and theory developed in \cite{L-WC}.
Throughout this section
the symbol $\Omega$ will always denote a nonempty open subset of $X$.
We denote by $\BV_c(\Om)$ the class of functions $\varphi\in\BV(\Om)$ with compact
support in $\Om$, that is, $\supp \varphi\Subset \Om$.

\begin{definition}
We say that $u\in\BV_{\loc}(\Om)$ is a $1$-minimizer  in $\Om$ (often
called function of least gradient) if
for all $\varphi\in \BV_c(\Om)$, we have
\begin{equation}\label{eq:definition of 1minimizer}
\Vert Du\Vert(\supp\varphi)\le \Vert D(u+\varphi)\Vert(\supp\varphi).
\end{equation}
We say that $u\in\BV_{\loc}(\Om)$ is a $1$-superminimizer in $\Om$
if \eqref{eq:definition of 1minimizer} holds for all nonnegative $\varphi\in \BV_c(\Om)$.
We say that $u\in\BV_{\loc}(\Om)$ is a $1$-subminimizer in $\Om$ if
\eqref{eq:definition of 1minimizer} holds for all nonpositive $\varphi\in \BV_c(\Om)$,
or equivalently if $-u$ is a $1$-superminimizer in $\Om$.
\end{definition}

Equivalently, we can replace $\supp\varphi$ by any set $A\Subset \Om$ containing $\supp\varphi$
in the above definitions.

If $\Om$ is bounded, and $\psi\colon\Om\to\overline{\R}$
and $f\in L^1_{\loc}(X)$
with $\Vert Df\Vert(X)<\infty$, we define the class of admissible functions
\[
\mathcal K_{\psi,f}(\Om):=\{u\in\BV_{\loc}(X):\,u\ge \psi\textrm{ in }\Om\textrm{ and }u=f\textrm{ in }X\setminus\Om\}.
\]
The (in)equalities above are understood in the a.e. sense.
For brevity, we sometimes write $\mathcal K_{\psi,f}$ instead of $\mathcal K_{\psi,f}(\Om)$.
By using a cutoff function,
it is easy to show that $\Vert Du\Vert(X)<\infty$
for every $u\in\mathcal K_{\psi,f}(\Om)$.

\begin{definition}
We say that $u\in\mathcal K_{\psi,f}(\Om)$ is a solution of the $\mathcal K_{\psi,f}$-obstacle problem
if $\Vert Du\Vert(X)\le \Vert Dv\Vert(X)$ for all $v\in\mathcal K_{\psi,f}(\Om)$.
\end{definition}

Whenever the characteristic function of a set $E$
is a solution of an obstacle problem,
for simplicity we will call $E$ a solution as well.
Similarly, if $\psi=\ch_A$ for some $A\subset X$, we let
$\mathcal K_{A, f}:=\mathcal K_{\psi, f}$.

Now we list some properties of superminimizers
and solutions of obstacle problems derived mostly in \cite{L-WC}.

\begin{lemma}[{\cite[Lemma 3.6]{L-WC}}]\label{lem:solutions from capacity}
	If $x\in X$, $0<r<R<\frac 18 \diam X$, and $A\subset B(x,r)$, then there exists
	$E\subset X$ that is a solution of the $\mathcal K_{A,0}(B(x,R))$-obstacle problem
	with
	\[
	P(E,X)\le \rcapa_1(A,B(x,R)).
	\]
\end{lemma}

\begin{proposition}[{\cite[Proposition 3.7]{L-WC}}]\label{prop:solutions are superminimizers}
If $u\in\mathcal K_{\psi,f}(\Om)$ is a solution
of the $\mathcal K_{\psi,f}$-obstacle problem, then $u$
is a $1$-superminimizer in $\Om$.
\end{proposition}

The following fact and its proof are similar to \cite[Lemma 3.2]{KKLS}.

\begin{lemma}\label{lem:subminimizer char}
Let $F\subset X$ with
$P(F,\Om)<\infty$ and suppose that for every $H\Subset \Om$, we have
\[
P(F,\Om)\le P(F\setminus H,\Om).
\]
Then $\ch_F$ is a $1$-subminimizer in $\Om$.
\end{lemma}
\begin{proof}
Take a nonnegative $\varphi\in\BV_c(\Om)$.
Observe that for every $0<s<1$, we have $\supp\{\varphi\ge s\}\Subset \Om$.
Thus by the coarea formula \eqref{eq:coarea},
\begin{align*}
\Vert D(\ch_F-\varphi)\Vert(\supp\varphi)
&\ge\int_0^1 P(\{\ch_F-\varphi>t\},\supp\varphi)\,dt\\
&=\int_0^1 P(F\setminus \{\varphi\ge 1-t\},\supp\varphi)\,dt\\
&\ge\int_0^1 P(F,\supp\varphi)\,dt=\Vert D\ch_F\Vert(\supp\varphi).
\end{align*}
\end{proof}

\begin{proposition}\label{prop:components are subminimizers}
Let $B(x,R)$ be a ball and let $F\subset X$ be a closed set with $P(F,X)<\infty$
and such that $\ch_F$ is a $1$-subminimizer in $B(x,R)$.
Denote the components of $F\cap \overline{B}(x,R)$
with nonzero $\mu$-measure by $F_1,F_2,\ldots$.
Then each $\ch_{F_k}$ is a $1$-subminimizer in $B(x,R)$.
\end{proposition}

\begin{proof}
Fix $k\in\N$ and take $H\Subset B(x,R)$.
We can assume that $H\subset F_k$ and that $P(F_k\setminus H,B(x,R))<\infty$.
Now
\begin{align*}
\sum_{\substack{j\in\N\\ j\neq k}}P(F_j,B(x,R))+P(F_k,B(x,R))
&=\sum_{j=1}^{\infty}P(F_j,B(x,R))\\
&= P(F,B(x,R))\quad\textrm{by Proposition }\ref{prop:connected components}\\
&\le P(F\setminus H,B(x,R))\\
&= \sum_{j=1}^{\infty}P(F_j\setminus H,B(x,R))\quad\textrm{by Proposition }\ref{prop:connected components}\\
&=\sum_{\substack{j\in\N\\ j\neq k}}P(F_j,B(x,R))+P(F_k\setminus H,B(x,R)).
\end{align*}
Note that since $\sum_{j=1}^{\infty}P(F_j,B(x,R))= P(F,B(x,R))<\infty$, we now get
\[
P(F_k,B(x,R))\le P(F_k\setminus H,B(x,R)).
\]
By Lemma \ref{lem:subminimizer char}, $\ch_{F_k}$ is a $1$-subminimizer in $B(x,R)$.
\end{proof}

We have the following weak Harnack inequality. We denote the positive
part of a function by $u_+:=\max\{u,0\}$.

\begin{theorem}[{\cite[Theorem 3.10]{L-WC}}]\label{thm:weak Harnack}
Suppose $k\in\R$ and $0<R<\tfrac 14 \diam X$ with $B(x,R)\Subset \Om$, and
assume either that
\begin{enumerate}[{(a)}]
\item $u$ is a $1$-subminimizer in $\Om$, or
\item $\Om$ is bounded, $u$ is a solution of the
$\mathcal K_{\psi, f}(\Om)$-obstacle problem,
and $\psi\le k$ a.e. in $B(x,R)$.
\end{enumerate}
Then for any $0<r<R$ and some constant $C_1=C_1(C_d,C_P)$,
\[
\esssup_{B(x,r)}u\le C_1\left(\frac{R}{R-r}\right)^{s}\vint{B(x,R)}(u-k)_+\,d\mu+k.
\]
\end{theorem}

For later reference, let us note that a close look at the proof
of the above theorem reveals that we can take
\begin{equation}\label{eq:C1}
C_1=2^{(s+1)^2}(6\widetilde{C}_S C_d)^s,
\end{equation}
where $\widetilde{C}$ is the constant from an $(s/(s-1),1)$-Sobolev inequality
with zero boundary values.

\begin{corollary}\label{cor:weak Harnack}
Suppose $k\in\R$, $x\in X$, $0<R<\tfrac 14 \diam X$, and
assume that $\ch_F$ is a $1$-subminimizer in $B(x,R)$ with $\mu(F\cap B(x,R/2))>0$.
Then
\[
\frac{\mu(B(x,R)\cap F)}{\mu(B(x,R))}\ge (2^s C_1)^{-1}.
\]
\end{corollary}
\begin{proof}
Let $0<\eps<R/2$.
Applying Theorem \ref{thm:weak Harnack}(i) with $\Om=B(x,R)$,
$u=\ch_F$, $k=0$, and $R/2,\, R-\eps$ in place of $r,\, R$, we get
\[
1\le C_1\left(\frac{R-\eps}{R-\eps-R/2}\right)^{s}\frac{\mu(B(x,R-\eps)\cap F)}{\mu(B(x,R-\eps))}.
\]
Letting $\eps\to 0$, we get the result.
\end{proof}

Recall the definitions of the lower and upper approximate limits
$u^{\wedge}$ and $u^{\vee}$ from \eqref{eq:lower approximate limit}
and \eqref{eq:upper approximate limit}.

\begin{theorem}[{\cite[Theorem 3.11]{L-WC}}]\label{thm:superminimizers are lsc}
Let $u$ be a $1$-superminimizer in $\Om$. Then $u^{\wedge}\colon\Om\to (-\infty,\infty]$
is lower semicontinuous.
\end{theorem}

\begin{lemma}\label{lem:smallness in annuli}
Let $B=B(x,R)$ be a ball with $0<R<\frac{1}{32} \diam X$, and
suppose that $W\subset B$.
Let $V\subset 4B$ be a solution of the $\mathcal K_{W,0}(4B)$-obstacle problem
(as guaranteed by Lemma \ref{lem:solutions from capacity}).
Then for all
$y\in 3 B\setminus 2 B$,
\[
\ch_V^{\vee}(y)\le C_2 R \frac{\rcapa_1(W,4B)}{\mu(B)}
\]
for some constant $C_2=C_2(C_d,C_P)$.
\end{lemma}

\begin{proof}
By Lemma \ref{lem:solutions from capacity} we know that
\[
P(V,X)\le \rcapa_1(W, 4B),
\]
and thus by the isoperimetric inequality \eqref{eq:isop inequality with zero boundary values},
\begin{equation}\label{eq:E1 has small measure}
\mu(V)\le 4C_S R P(V,X)\le 4C_S R \rcapa_1(W,4B).
\end{equation}
For any $z\in 3 B\setminus 2 B$ we have
$B(z,R)\subset 4 B\setminus B$. Since now
 $W\cap B(z,R)=\emptyset$, we can apply
Theorem \ref{thm:weak Harnack}(b) with $k=0$ to get
\begin{align*}
\sup_{B(z,R/2)} \ch_V^{\vee }
&\le \esssup_{B(z,R/2)}\ch_V\\
&\le C_1\left(\frac{R}{R-R/2}\right)^s\vint{B(z,R)}(\ch_V)_+\,d\mu\\
&= \frac{2^s C_1}{\mu(B(z,R))}\int_{B(z,R)} (\ch_V)_+\,d\mu\\
&\le \frac{2^s C_1 C_d^2}{\mu(B)}\mu(V)\\
&\le 2^{s+2} C_1 C_d^2  C_S R \frac{\rcapa_1(W,4B)}{\mu(B)}\quad\textrm{by }\eqref{eq:E1 has small measure}.
\end{align*}
Thus we can choose
$C_2=2^{s+2} C_1 C_d^2  C_S$.
\end{proof}

\section{Constructing a ``geodesic'' space}\label{sec:constructing a quasiconvex space}

In this section we construct a suitable space where the Mazurkiewicz metric
agrees with the ordinary one; this space will be needed
in the proof of the main result.

Recall that in Section
\ref{sec:strong boundary points}, in the space $(Z,\widehat{d},\mu)$
we defined the Mazurkiewicz metric
$\widehat{d}_M$; given a set $V\subset X$ we now define
\[
d_{M}^V(x,y):=\inf\{\diam K: K\subset X\setminus V\textrm{ is a continuum containing }x,y\},
\quad x,y\in X\setminus V.
\]
If $V=\emptyset$, we leave it out of the notation, consistent with
\eqref{eq:widehat d c}.

\begin{lemma}\label{lem:new metric lemma}
Let $V\subset X$ be a bounded open set and let $B(x_0,R_0)$ be a ball such
that $V\Subset B(x_0,R_0)$, and
$\overline{B}(x_0,R_0)\setminus V$ is connected.
Moreover, suppose there is $R>0$ such that
for every $x\in X\setminus V$ and $0<r\le R$,
the connected components of $\overline{B}(x,r)\setminus V$
intersecting $B(x,r/2)$ are finite in number.

Then $d_M^V$ is a metric on $X\setminus V$ such that $d\le d_M^V$,
$d_M^V$ induces the same topology on $X\setminus V$ as $d$, $(d_M^V)_M=d_M^V$,
and $(X\setminus V,d_M^V)$ is complete.
\end{lemma}

Note that explicitly, for $x,y\in X\setminus V$,
\[
(d_{M}^V)_M(x,y)=
\inf\{\diam_{d_M^V} K:\, K\subset X\setminus V
\textrm{ is a }d_M^V\textrm{-continuum containing }x,y\}.
\]
\begin{proof}
Since $V\Subset B(x_0,R_0)$ and $\overline{B}(x_0,R_0)\setminus V$ is connected,
also every $\overline{B}(x_0,r)\setminus V$ with $r\ge R_0$ is connected,
by the fact that $X$ is geodesic. Thus we have for all $x,y\in X\setminus V$
\[
d_M^V(x,y)\le 2\max\{R_0,d(x,x_0),d(y,x_0)\}<\infty.
\]
Obviously $d\le d_M^V$ and
$d_M^V(x,x)=0$ for all $x\in X\setminus V$.
If $d_M^V(x,y)=0$ then $d(x,y)=0$ and so $x=y$. Obviously
also $d_M^V(x,y)=d_M^V(y,x)$ for all $x,y\in X\setminus V$.
Finally, take $x,y,z\in X\setminus V$.
Take a continuum $K_1\subset X\setminus V$ containing $x,y$ and a continuum
$K_2\subset X\setminus V$ containing $y,z$.
Then $K_1\cup K_2\subset X\setminus V$ is a continuum containing $x,z$
and so
\[
d_M^V(x,z)\le \diam(K_1\cup K_2)\le \diam(K_1)+\diam (K_2).
\]
Taking infimum over $K_1$ and $K_2$, we conclude that the triangle inequality holds.
Hence $d_M^V$ is a metric on $X\setminus V$.

To show that the topologies induced on $X\setminus V$ by $d$ and $d_M^V$ are the same,
take a sequence $x_j\to x$ with respect to $d$ in $X\setminus V$.
Fix $\eps\in (0,R)$. Consider the
components of $\overline{B}(x,\eps/2)\setminus V$ intersecting $B(x,\eps/4)$.
By assumption there are only finitely many.
Each of them not containing $x$ is at a nonzero distance from $x$ and
so for large $j$, every $x_j$ belongs to the component containing $x$; denote it $F_1$.
For such $j$, we have
\[
d_M^V(x_j,x)\le \diam F_1\le \eps. 
\]
We conclude that $x_j\to x$ also with respect to $d_M^V$.
Since we had $d\le d_M^V$, it follows that the topologies are the same.

If $x,y\in X\setminus V$, and $\eps>0$, we can take a continuum $K$
containing $x$ and $y$, with $\diam K<d_M^V(x,y)+\eps$.
The set $K$ is still a continuum in the metric space $(X\setminus V,d_M^V)$, and
for every $z,w\in K$,
\[
d_M^V(z,w)\le \diam K< d_M^V(x,y)+\eps.
\]
It follows  that $\diam_{d_M^V} K\le d_M^V(x,y)+\eps$,
and so $(d_M^V)_M(x,y)\le d_M^V(x,y)+\eps$, showing that $(d_M^V)_M=d_M^V$.

Finally let $(x_j)$ be a Cauchy sequence in $(X\setminus V, d_M^V)$.
Since $d\le d_M^V$, it is also a Cauchy sequence in $(X,d)$,
and so $x_j\to x\in X\setminus V$ with respect to $d$.
But as we showed before, this implies that $x_j\to x$
with respect to $d_M^V$.
\end{proof}

Let $B$ be a ball and let $B_1,B_2\subset B$
be two other balls, and let $u\in L^1(B)$ such that $u=1$ in $B_1$
and $u=0$ in $B_2$. Then we have
\begin{equation}\label{eq:KoLa result}
\int_{B}|u-u_{B}|\,d\mu\ge \frac{1}{2}\min\{\mu(B_1),\mu(B_2)\};
\end{equation}
this follows easily by considering the cases $u_{B}\le 1/2$
and $u_{B}\ge 1/2$.

We have the following \emph{linear local connectedness};
versions of this property have been proved before e.g. in \cite{HK}, but they
assume certain growth bounds on the measure, which we do not want to assume.

\begin{lemma}\label{lem:lin loc connectedness}
Let $B(x_0,R)$ be a ball and
let $V\subset B(x_0,2R)$ with
\begin{equation}\label{eq:V capacity in 4B}
\rcapa_1(V,B(x_0,3R))<\frac{1}{12C_P C_d^3}\frac{\mu(B(x_0,R))}{R}.
\end{equation}
Then every pair of points  $y,z\in B(x_0,5R)\setminus B(x_0,4R)$
can be joined by a curve in $B(x_0,6R)\setminus V$.
\end{lemma}
\begin{proof}
If $d(y,z)\le 2R$, then the result is clear since the space is geodesic.
Thus assume that $d(y,z)> 2R$. Consider the disjoint balls
$B_1:=B(y,R)$ and $B_2:=B(z,R)$ which both belong to
$B(x_0,6R)\setminus B(x_0,3R)$.
Denote by $\Gamma$ the family of curves $\gamma$ in $B(x_0,6R)$
with $\gamma(0)\in B_1$ and $\gamma(\ell_{\gamma})\in B_2$.
Note that $\Mod_1(\Gamma)<\infty$ since $\dist(B_1,B_2)>0$.
Let $\eps>0$.
Let $g\in L^1(B(x_0,6R))$ such that $\int_{\gamma}g\,ds\ge 1$
for all $\gamma\in\Gamma$ and
\[
\int_{B(x_0,6R)}g\,d\mu<\Mod_1(\Gamma)+\eps.
\]
Let
\[
u(x):=\min\left\{1,\inf \int_{\gamma}g\,ds\right\},\quad x\in B(x_0,6R),
\]
where the infimum is taken over curves $\gamma$ (also constant curves)
in $B(x_0,6R)$ with $\gamma(0)=x$ and
$\gamma(\ell_{\gamma})\in B_1$. Then $u=1$ in $B_2$.
Moreover, $g$ is an upper gradient of $u$ in $B(x_0,6R)$, see \cite[Lemma 5.25]{BB},
and $u$ is $\mu$-measurable by \cite[Theorem 1.11]{JJRRS}.
In total, $u\in N^{1,1}(B(x_0,6R))$ with $u=0$ in $B_1$ and
$u=1$ in $B_2$.
Thus using the Poincar\'e inequality,
\begin{align*}
\Mod_1(\Gamma)
&>\int_{B(x_0,6R)}g\,d\mu-\eps\\
&\ge \frac{1}{6C_P R}\int_{B(x_0,6R)}|u-u_{B(x_0,6R)}|\,d\mu-\eps\\
&\ge \frac{1}{12C_P R}\min\{\mu(B_1),\mu(B_2)\}-\eps\quad\textrm{by }\eqref{eq:KoLa result}\\
&\ge \frac{1}{12C_P C_d^3 R}\mu(B(x_0,R))-\eps
\end{align*}
and so
\[
\Mod_1(\Gamma)\ge \frac{1}{12C_P C_d^3 R}\mu(B(x_0,R)).
\]
On the other hand, by \eqref{eq:V capacity in 4B} we find a function
$v\in N^{1,1}(X)$ such that $v=1$ in $V$,
$v=0$ in $X\setminus B(x_0,3R)$,
and $v$ has an upper gradient $\widetilde{g}$ satisfying
\[
\int_X \widetilde{g}\,d\mu< \frac{1}{12C_P C_d^3}\frac{\mu(B(x_0,R))}{R}.
\]
Denote the family of all curves intersecting $V$ by $\Gamma_V$.
Now $\int_{\gamma}\widetilde{g}\,ds\ge 1$ for all $\gamma\in \Gamma\cap \Gamma_V$,
and so
\[
\Mod_1(\Gamma\cap \Gamma_V)\le \int_X \widetilde{g}\,d\mu
< \frac{1}{12C_P C_d^3}\frac{\mu(B(x_0,R))}{R}.
\]
Thus $\Gamma\setminus \Gamma_V$ is nonempty.
Take a curve $\gamma\in \Gamma\setminus \Gamma_V$. Now we get the required curve
by concatenating three curves:
the first going from $y$ to $\gamma(0)$ inside $B(y,R)$
(using the fact that the space is geodesic), the second $\gamma$, and the third 
going from $\gamma(\ell_{\gamma})$ to $z$ inside $B(z,R)$.
\end{proof}

By using an argument involving Lipschitz cutoff functions, it is easy to
see that for any ball $B(x,r)$ and any set $A\subset B(x,r)$, we have
\begin{equation}\label{eq:capacity and Hausdorff measure}
\rcapa_1(A,B(x,3r))\le C_d \mathcal H(A).
\end{equation}

In the following proposition we construct the space in which the metric and
Mazurkiewicz metric agree.

\begin{proposition}\label{prop:constructing the quasiconvex space}
Let $B=B(x,R)$ be a ball with $0<R<\frac{1}{32} \diam X$,
and let $A\subset B$ with
\[
\mathcal H(A)
\le \frac{1}{24 C_P C_S C_2 C_r C_d^4}
\frac{\mu(B)}{R}.
\]
Let $\eps>0$. Then we find an open set $V$ with
$A\subset V\subset 2B$ and
\[
P(V,X)\le C_d\mathcal H(A)+\eps,
\]
and such that the following hold:
the space $(Z,d_M^V,\mu)$ with $Z=X\setminus V$ is a complete metric space with
$(d_M^V)_M=d_M^V$, $\mu$ in $Z$ is a Borel regular outer measure
and doubling with constant $2^s C_1 C_d^2$, and
for every $y\in X\setminus V$ and $r>0$ we have
\[
\frac{\mu(B_Z(y,r))}{\mu(B(y,r))}\ge (2^s C_1 C_d)^{-1}
\]
where $B_Z(y,r)$ denotes an open ball in $Z$, defined with respect to the metric
$d_M^V$.
\end{proposition}
\begin{proof}
Using  the fact that $\rcapa_1$
is an outer capacity in the sense of \eqref{eq:rcapa outer capacity},
as well as \eqref{eq:capacity and Hausdorff measure},
we find an open set $W$, with $A\subset W\subset B$, such that (note that 
the first inequality is obvious)
\[
\rcapa_1(W,4B)\le \rcapa_1(W,3B)\le \rcapa_1(A,3B)+\eps\le C_d\mathcal H(A)+\eps.
\]
We can assume that
\[
\eps<\frac{1}{24 C_P C_S C_2 C_r C_d^3}\frac{\mu(B)}{R}.
\]
Take a solution $V$ of the
$\mathcal K_{W,0}(4B)$-obstacle problem.
By Lemma \ref{lem:solutions from capacity}, we have
\[
P(V,X)\le \rcapa_1(W,4B)\le C_d\mathcal H(A)+\eps.
\]
By Theorem \ref{thm:superminimizers are lsc}, the function $\ch_V^{\wedge}$
is lower semicontinuous, and by redefining $V$ in a set of measure zero,
we get $\ch_V=\ch_V^{\wedge}$ and so $V$ is open.
By Lemma \ref{lem:smallness in annuli} we know that
for all
$y\in 3 B\setminus 2B$,
\[
\ch_V^{\vee}(y)\le C_2 R \frac{\rcapa_1(W,4B)}{\mu(B)}
\le C_2 R \frac{C_d \mathcal H(A)+\eps}{\mu(B)}<1
\]
and so $\ch_V^{\vee}=0$ in $3 B\setminus 2B$. Then
in fact $\ch_V=\ch_V^{\vee}=0$ in $4B\setminus 2B$, that is,
$V\subset 2B$, because else we could remove the parts of $V$ inside $4B\setminus 3B$
to decrease $P(V,X)$.

By the isoperimetric inequality
\eqref{eq:isop inequality with zero boundary values},
\begin{equation}\label{eq:measure of V}
\mu(V)\le 2C_S R P(V,X)\le 2 C_S C_d R \mathcal H(A)
+2C_S R \eps
\le \frac{\mu(B)}{2C_d^2}.
\end{equation}
Moreover, by \eqref{eq:variational one and BV capacity} we get
\begin{align*}
\rcapa_1(V,3B)
&\le C_r \rcapa_{\BV}^{\vee}(V,3B)\\
&\le C_r P(V,X)\le C_r C_d \mathcal H(A)+C_r\eps
< \frac{1}{12C_P C_d^3}\frac{\mu(B)}{R}.
\end{align*}
By Lemma \ref{lem:lin loc connectedness},
$5B\setminus 4B$ belongs to one component of
$6\overline{B}\setminus V$.
Since the space is geodesic, in fact
$6\overline{B}\setminus 4B$ belongs to one component of
$6\overline{B}\setminus V$.
Call this component $F_1$. Moreover, denote $F:=X\setminus V$; $F$ is a closed set with
$P(F,X)=P(V,X)<\infty$.
Consider all components of $F\cap 6\overline{B}$.
Suppose there is another component $F_2$ with nonzero $\mu$-measure.
Denote by $F_1,F_2,\ldots$ all the components with nonzero $\mu$-measure
(as usual, some of these may be empty).
By the relative isoperimetric inequality
\eqref{eq:relative isoperimetric inequality}, we have
\begin{equation}\label{eq:F2 has perimeter}
P(F_2,6B)>0.
\end{equation}
Now the set $\widetilde{V}:=V\cup \bigcup_{j=2}^{\infty}F_j\subset 4B$ is admissible
for the $\mathcal K_{W,0}(4B)$-obstacle problem, with
\begin{align*}
P(\widetilde{V},X)
&=P(\widetilde{V},6B)\\
&=P\Bigg(X\setminus \Big(V\cup \bigcup_{j=2}^{\infty}F_j\Big),6B\Bigg)\\
&=P\Bigg(F\setminus \bigcup_{j=2}^{\infty}F_j,6B\Bigg)\\
&= P(F,6B)-\sum_{j=2}^{\infty}P(F_j,6B)\quad\textrm{by Proposition }\ref{prop:connected components}\\
&<P(F,6B)\quad\textrm{by }\eqref{eq:F2 has perimeter}\\
&=P(V,6B)=P(V,X).
\end{align*}
This is a contradiction with the fact that $V$ is a solution
of the $\mathcal K_{W,0}(4B)$-obstacle problem. Thus
by Proposition \ref{prop:connected components},
$F\cap 6\overline{B}$
is the union of $F_1$ and a set of measure zero $N$.
Suppose
\[
y\in 6\overline{B}\cap F\setminus F_1
=4B\cap F\setminus F_1.
\]
Now $y$ is at a nonzero distance from $F_1$. Thus for small $\delta>0$,
\[
\mu(B(y,\delta)\cap F)\le \mu(N)=0.
\]
Note that since we had  $\ch_V=\ch_V^{\wedge}$, it follows that
$\ch_F=\ch_F^{\vee}$. Thus in fact such $y$ cannot exist
and $F\cap 6\overline{B}=F_1$ is connected.

If $y\in F\setminus B(x,3R)$ and
$0<r\le R$, then $\overline{B}(y,r)\cap F=\overline{B}(y,r)$ is connected since
the space is geodesic.
If $y\in F\cap B(x,3R)$ and
$0<r\le R$, by Proposition \ref{prop:connected components} we know that
$F\cap \overline{B}(y,r)$
consists of at most countably many components $F_1,F_2,\ldots$
and a set of measure zero $\widetilde{N}$.
By Proposition \ref{prop:solutions are superminimizers}
we know that $\ch_F$ is a $1$-subminimizer in $B(x,4R)$, and then 
also in $B(y,r)\subset B(x,4R)$.
Then each $\ch_{F_j}$ is a $1$-subminimizer in $B(y,r)$
by Proposition \ref{prop:components are subminimizers}.
By Corollary \ref{cor:weak Harnack} we get for each $F_j$
with $\mu(B(y,r/2)\cap F_j)>0$ that
\begin{equation}\label{eq:subminimizer component measure lower bound}
\frac{\mu(F_j\cap B(y,r))}{\mu(B(y,r))}\ge (2^s C_1)^{-1}.
\end{equation}
Thus there are less than $2^s C_1+1$ such components, which we can
relabel $F_1,\ldots,F_M$.
Suppose
\[
z\in B(y,r/2)\cap \widetilde{N}\setminus \bigcup_{j=1}^M F_j.
\]
This is at nonzero distance from all $F_1,\ldots,F_M$. Thus for small $\delta>0$,
\[
\mu(B(z,\delta)\cap F)\le \mu(\widetilde{N})+\sum_{j=M+1}^{\infty}\mu(F_j\cap B(y,r/2))=0.
\]
As before, we have $\ch_F=\ch_F^{\vee}$. Thus in fact such $z$ cannot exist
and
\[
F\cap B(y,r/2)=B(y,r/2)\cap \bigcup_{j=1}^M F_j.
\]
Now Lemma \ref{lem:new metric lemma} gives that $(Z,d_M^V,\mu)$,
with $Z=X\setminus V$,
is a complete metric space, $d\le d_M^V$, the topologies induced
by $d$ and $d_M^V$ are the same, and $(d_M^V)_M=d_M^V$.
Note that $\mu$ restricted to the subsets of $X\setminus V$ is still a Borel
regular outer measure, see \cite[Lemma 3.3.11]{HKST15}.
Since the topologies induced by $d$ and $d_M^V$
are the same, $\mu$ remains a Borel regular outer measure in $Z$.
(Note that as sets, we have $X\setminus V=F=Z$.)

Denoting by $F_1$ the component of $F\cap \overline{B}(y,r)$ containing
$y$,
by \eqref{eq:subminimizer component measure lower bound} we have for
$y\in F\cap B(x,3R)$ and
$0<r\le R$ that
\begin{equation}\label{eq:size of F1}
\frac{\mu(B(y,r)\cap F_1)}{\mu(B(y,r))}\ge (2^s C_1)^{-1}.
\end{equation}
Recall that if $y\in F\setminus B(x,3R)$, then $F_1=\overline{B}(y,r)$
and so \eqref{eq:size of F1} holds.
Eq. \eqref{eq:size of F1} is easily seen to hold also for
all $x\in F$ and $r>R$ by \eqref{eq:measure of V}.
It follows that for all $y\in F$ and $r>0$, we have
\[
\frac{\mu(B_Z(y,2r))}{\mu(B(y,r))}\ge (2^s C_1)^{-1}
\]
and so in fact
\[
\frac{\mu(B_Z(y,r))}{\mu(B(y,r))}\ge (2^s C_1 C_d)^{-1}\quad\textrm{for all }y\in Z
\textrm{ and }r>0,
\]
as desired.
Thus 
\[
\frac{\mu(B_Z(y,2r))}{\mu(B_Z(y,r))}\le 2^s C_1 C_d\frac{\mu(B(y,2r))}{\mu(B(y,r))}
\le 2^s C_1 C_d^2.
\]
Thus in the space $(Z,d_M^V,\mu)$, the measure $\mu$
is doubling with constant $2^s C_1 C_d^2$.
\end{proof}

\section{Proof of the main result}\label{sec:proof of the main result}

In this section we prove the main result of the paper, Theorem \ref{thm:main theorem}.

First note that with the choice
$\widehat{C}_d=2^s C_1 C_d^2$, the constant appearing in
Corollary \ref{cor:density points} becomes
\[
\frac{1}{4 \widehat{C}_d^{12}}
=\frac{1}{4 (2^s C_1 C_d^2)^{12}}=:\beta_0.
\]
Recall from \eqref{eq:C1} that we can take
$C_1=2^{(s+1)^2}(6\widetilde{C}_S C_d)^s$.
Define
\begin{equation}\label{eq:definition of beta}
\begin{split}
\beta:=\frac{\beta_0}{2^s C_1 C_d}=\frac{1}{2^{2+s}C_1 C_d (2^s C_1 C_d^2)^{12}}
&=\frac{1}{2^{13s +2} (2^{(s+1)^2}(6\widetilde{C}_S C_d)^s)^{13} C_d^{25}}\\
&=\frac{1}{2^{13s^2+52s +15}3^{13s} \widetilde{C}_S^{13s} C_d^{13s+25}}.
\end{split}
\end{equation}
Note that in the Euclidean space $\R^n$, $n\ge 2$, we can take $C_d=2^n$, $s=n$, and $\widetilde{C}_S=2^{-1}n^{-1/2}\omega_n^{1/n}$, where $\omega_n$ is the volume
of the Euclidean unit ball, and then
\begin{equation}\label{eq:choice of beta in Euclidean space}
\beta = 2^{-26n^2-64n-15} 3^{-13n}n^{13n/2}\omega_n^{-13}.
\end{equation}

Recall the definition of the strong boundary from \eqref{eq:strong boundary}.

\begin{theorem}\label{thm:comparison of boundaries}
Let $\Om\subset X$ be open and let
$E\subset X$ be $\mu$-measurable with
$\mathcal H(\Sigma_{\beta} E\cap \Om)<\infty$.
Then $\mathcal H((\partial^*E\setminus \Sigma_{\beta} E)\cap \Om)=0$.
\end{theorem}

\begin{proof}

By a standard covering argument (see e.g. the proof of \cite[Lemma 2.6]{KKST3}),
we find that
\[
\lim_{r\to 0}r\frac{\mathcal H(\Sigma_{\beta} E\cap B(x,r))}{\mu(B(x,r))}=0
\]
for all $x\in \Om\setminus (\Sigma_{\beta}E\cup N)$, with $\mathcal H(N)=0$.
We will show that $\partial^*E\cap \Om\subset (\Sigma_{\beta} E\cup N)\cap \Om$
and thereby prove the claim.

Suppose instead that there exists $x\in\Om\cap \partial^*E\setminus (\Sigma_{\beta} E\cup N)$.
Then
\[
\lim_{r\to 0}r\frac{\mathcal H(\Sigma_{\beta} E\cap B(x,r))}{\mu(B(x,r))}=0
\]
and
\[
\limsup_{r\to 0}\frac{\mu(B(x,r)\cap E)}{\mu(B(x,r))}>0
\quad \textrm{and}\quad\limsup_{r\to 0}\frac{\mu(B(x,r)\setminus E)}{\mu(B(x,r))}>0.
\]
Thus for some $0<a<(2C_d^2)^{-1}$ we have
\[
\limsup_{r\to 0}\frac{\mu(B(x,r)\cap E)}{\mu(B(x,r))}>C_d a
\quad \textrm{and}\quad\limsup_{r\to 0}\frac{\mu(B(x,r)\setminus E)}{\mu(B(x,r))}>C_d a.
\]
Now we can choose $0<R_0<\tfrac{1}{32} \diam X$ such that
\[
\frac{\mu(B(x,40^{-1}R_0)\cap E)}{\mu(B(x,40^{-1}R_0))}>a
\]
and
\[
r\frac{\mathcal H(\Sigma_{\beta} E\cap B(x,r))}{\mu(B(x,r))}
<\frac{a}{24 \cdot 2^s C_P C_S C_1 C_2 C_r C_d^8}
\]
for all $0<r\le R_0$.
Choose the smallest $j=0,1,\ldots$ such that for some
$r\in (2^{-j-1}R_0,2^{-j}R_0]$ we have
\[
\frac{\mu(B(x,40^{-1}r)\setminus E)}{\mu(B(x,40^{-1}r))}>C_d a
\quad\textrm{and thus}\quad\frac{\mu(B(x,40^{-1}2^{-j}R_0)\setminus E)}{\mu(B(x,40^{-1}2^{-j}R_0))}>a.
\]
Let $R:=2^{-j}R_0$. If $j\ge 1$, then
\[
\frac{\mu(B(x,20^{-1}R)\setminus E)}{\mu(B(x,20^{-1}R))} \le C_d a 
\]
and so
\begin{align*}
\frac{\mu(B(x,40^{-1}R)\cap E)}{\mu(B(x,40^{-1}R))}
&\ge \frac{\mu(B(x,40^{-1}R))-\mu(B(x,20^{-1}R)\setminus E)}{\mu(B(x,40^{-1}R))}\\
&\ge 1 -C_d \frac{\mu(B(x,20^{-1}R)\setminus E)}{\mu(B(x,20^{-1}R))}\\
&\ge 1 -C_d^2 a\ge 1 -C_d^2 \frac{1}{2 C_d^2}=\frac{1}{2}> a.
\end{align*}
Thus
\begin{equation}\label{eq:portion of E}
a<\frac{\mu(B(x,40^{-1}R)\cap E)}{\mu(B(x,40^{-1}R))}<1-a,
\end{equation}
which holds clearly also if $j=0$, and
\[
R\frac{\mathcal H(\Sigma_{\beta} E\cap B(x,R))}{\mu(B(x,R))}
<\frac{a}{24 \cdot 2^s C_P C_S C_1 C_2 C_r C_d^8}.
\]
Let $A:= \Sigma_{\beta} E\cap B(x,R)$.
By Proposition \ref{prop:constructing the quasiconvex space}
we find an open set $V$ with $A\subset V\subset B(x,2R)$ and such that denoting $Z=X\setminus V$, the space
$(Z,d_M^V,\mu)$ is a complete metric space with $d\le d_M^V=(d_M^V)_M$ in $Z$,
$\mu$ in $Z$ is a Borel regular outer measure and
doubling with constant $\widehat{C}_d=2^s C_1 C_d^2$,
and for every $y\in Z$ and $r>0$ we have
\begin{equation}\label{eq:lower measure property}
\frac{\mu(B_Z(y,r))}{\mu(B(y,r))}\ge (2^s C_1 C_d)^{-1}.
\end{equation}
Moreover, by choosing a suitably small $\eps>0$,
\begin{equation}\label{eq:perimeter of V}
P(V,X)\le C_d\mathcal H(A)+\eps 
< \frac{a}{2^{s+1}C_P C_S C_1 C_d^7}\frac{\mu(B(x,R))}{R}.
\end{equation}
Thus by the isoperimetric inequality
\eqref{eq:isop inequality with zero boundary values},
\[
\mu(V)\le 2C_S RP(V,X)< \frac{1}{C_d^6}\mu(B(x,R))
\le \mu(B(x,40^{-1}R)).
\]
Thus we can choose $y\in B(x,40^{-1}R)\setminus V$.
Denote $F:=X\setminus V$.
Let $F_1$ be the component of
$\overline{B}(y,20^{-1}R)\setminus V$ containing $y$.
By \eqref{eq:size of F1} (and the comments after it) we know that
\[
\mu(F_1)\ge (2^s C_1)^{-1}\mu(B(y,20^{-1}R)).
\]
Since $\mu(\{z\in X:\,d(z,y)=20^{-1}R\})=0$ (see \cite[Corollary 2.2]{Buc}),
now also
\[
\mu(B(y,20^{-1}R)\cap F_1)\ge (2^s C_1)^{-1}\mu(B(y,20^{-1}R)).
\]
Suppose that
\[
\mu(B(y,20^{-1}R)\setminus F_1)\ge \frac{a}{2^s C_1 C_d^2}\mu(B(y,20^{-1}R)).
\]
Then
\begin{align*}
P(V,B(y,20^{-1}R))
&=P(F,B(y,20^{-1}R))\\
&\ge P(F_1,B(y,20^{-1}R))\quad\textrm{by Proposition }\ref{prop:connected components}\\
&\ge \frac{a}{2\cdot 2^s C_P C_1 C_d^2}
\frac{\mu(B(y,20^{-1}R))}{20^{-1}R}
\quad\textrm{by }\eqref{eq:relative isoperimetric inequality}\\
&\ge \frac{a}{2^{s+1}  C_P C_1 C_d^7}\frac{\mu(B(x,R))}{R}.
\end{align*}
This contradicts \eqref{eq:perimeter of V}, and so necessarily
\begin{equation}\label{eq:complement of F1 small measure}
\mu(B(y,20^{-1}R)\setminus F_1)< \frac{a}{2^s C_1 C_d^2}\mu(B(y,20^{-1}R))
\le \frac{a}{C_d^2}\mu(B(y,20^{-1}R)).
\end{equation}
Now
\begin{align*}
C_d\frac{\mu(B_{Z}(y,10^{-1}R)\cap E)}{\mu(B(y,10^{-1}R))}
&\ge \frac{\mu(B(y,20^{-1}R)\cap E\cap F_1)}{\mu(B(y,20^{-1}R))}\\
&\ge \frac{\mu(B(y,20^{-1}R)\cap E)}{\mu(B(y,20^{-1}R))}-\frac{a}{C_d^2}\quad\textrm{by }\eqref{eq:complement of F1 small measure}\\
&\ge \frac{1}{C_d^2}\frac{\mu(B(x,40^{-1}R)\cap E)}{\mu(B(x,40^{-1}R))}-\frac{a}{C_d^2}\\
&> \frac{a}{C_d^2}-\frac{a}{C_d^2}= 0\quad\textrm{by }\eqref{eq:portion of E}.
\end{align*}
The same string of inequalities holds with $E$ replaced by $X\setminus E$.
It follows that
\[
0<\mu(B_{Z}(y,10^{-1}R)\cap E)<\mu(B_Z(y,10^{-1}R)).
\]
Denoting by $\Sigma_{\beta_0}^ZE$ the strong boundary defined in the space
$(Z,d_M^V,\mu)$, by Corollary \ref{cor:density points} we find a point
\[
z\in\Sigma_{\beta_0}^Z E\cap B_{Z}(y,9R/10)
\subset \Sigma_{\beta_0}^Z E\cap B(y,9R/10)\setminus V
\subset \Sigma_{\beta_0}^ZE\cap B(x,R)\setminus V.
\]
Now using \eqref{eq:lower measure property}, we get
\[
\liminf_{r\to 0}\frac{\mu(B(z,r)\cap E)}{\mu(B(z,r))}
\ge \liminf_{r\to 0}\frac{\mu(B_{Z}(z,r)\cap E)}{\mu(B_{Z}(z,r))}
\frac{\mu(B_{Z}(z,r))}{\mu(B(z,r))}
\ge \beta_0\frac{1}{2^s C_1 C_d}=\beta,
\]
and analogously for $X\setminus E$.
Thus $z\in\Sigma_{\beta} E\cap B(x,R)\setminus V$,
a contradiction.
\end{proof}

Recall the usual version of Federer's characterization in metric spaces.

\begin{theorem}[{\cite[Theorem 1.1]{L-Fedchar}}]\label{thm:Federers characterization}
Let $\Om\subset X$ be an open set, let $E\subset X$ be a $\mu$-measurable set, and
suppose that $\mathcal H(\partial^*E\cap \Om)<\infty$. Then $P(E,\Om)<\infty$.
\end{theorem}

Now we can prove our main result; recall from the discussion on page
\pageref{quasiconvex and geodesic} that one can assume the space to be
geodesic, as we have done in most of the paper.
(However, the constant $\beta$, which is defined explicitly in geodesic spaces
in \eqref{eq:definition of beta}, will have a different form in the original
space considered in Theorem \ref{thm:main theorem}.)

\begin{proof}[Proof of Theorem \ref{thm:main theorem}]
By Theorem \ref{thm:comparison of boundaries} we get $\mathcal H(\partial^*E\cap \Om)<\infty$, and then Theorem \ref{thm:Federers characterization} gives
$P(E,\Om)<\infty$.
\end{proof}

\noindent Address:\\

\noindent Institut f\"ur Mathematik\\
Universit\"at Augsburg\\
Universit\"atsstr. 14\\
86159 Augsburg, Germany\\
E-mail: {\tt panu.lahti@math.uni-augsburg.de}


\begin{thebibliography}{ACMM}

\bibitem{A1}L. Ambrosio,
\textit{Fine properties of sets of finite perimeter in doubling metric measure spaces},
Calculus of variations, nonsmooth analysis and related topics.
Set-Valued Anal. 10 (2002), no. 2-3, 111--128.

\bibitem{AFP}L. Ambrosio, N. Fusco, and D. Pallara,
\textit{Functions of bounded variation and free discontinuity problems.}
Oxford Mathematical Monographs. The Clarendon Press, Oxford University Press, New York, 2000.

\bibitem{AMP}L. Ambrosio, M. Miranda, Jr., and D. Pallara,
\textit{Special functions of bounded variation in doubling metric measure spaces},
Calculus of variations: topics from the mathematical heritage of E. De Giorgi, 1--45,
Quad. Mat., 14, Dept. Math., Seconda Univ. Napoli, Caserta, 2004.

\bibitem{BB}A. Bj\"orn and J. Bj\"orn,
\textit{Nonlinear potential theory on metric spaces},
EMS Tracts in Mathematics, 17. European Mathematical Society (EMS), Z\"urich, 2011. xii+403 pp.

\bibitem{Buc}S. M. Buckley,
\textit{Is the maximal function of a Lipschitz function continuous?},
Ann. Acad. Sci. Fenn. Math. 24 (1999), no. 2, 519--528. 

\bibitem{Chl}M. Chleb\'ik,
\textit{Going beyond variation of sets},
Nonlinear Anal. 153 (2017), 230--242. 

\bibitem{EvGa}L. C. Evans and R. F. Gariepy,
\textit{Measure theory and fine properties of functions},
Studies in Advanced Mathematics series, CRC Press, Boca Raton, 1992.

\bibitem{Fed}H. Federer,
\textit{Geometric measure theory},
Die Grundlehren der mathematischen Wissenschaften, Band 153 Springer-Verlag New York Inc., New York 1969 xiv+676 pp.

\bibitem{Giu84}E. Giusti,
\textit{Minimal surfaces and functions of bounded variation},
Monographs in Mathematics, 80. Birkh\"auser Verlag, Basel, 1984. xii+240 pp.

\bibitem{Hj}P. Haj\l{}asz,
\textit{Sobolev spaces on metric-measure spaces},
Heat kernels and analysis on manifolds, graphs, and metric spaces (Paris, 2002), 173--218,
Contemp. Math., 338, Amer. Math. Soc., Providence, RI, 2003.

\bibitem{HaKi}H. Hakkarainen and J. Kinnunen,
\textit{The BV-capacity in metric spaces},
Manuscripta Math. 132 (2010), no. 1-2, 51--73.

\bibitem{Hei}J. Heinonen,
\textit{Lectures on analysis on metric spaces},
Universitext. Springer-Verlag, New York, 2001. x+140 pp.

\bibitem{HK}J. Heinonen and P. Koskela,
\textit{Quasiconformal maps in metric spaces with controlled geometry},
Acta Math. 181 (1998), no. 1, 1--61.

\bibitem{HKST15}J. Heinonen, P. Koskela, N. Shanmugalingam, and J. Tyson,
\textit{Sobolev spaces on metric measure spaces.
An approach based on upper gradients},
New Mathematical Monographs, 27. Cambridge University Press, Cambridge, 2015. xii+434 pp.

\bibitem{JJRRS}E. J\"arvenp\"a\"a, M. J\"arvenp\"a\"a, K. Rogovin, S. Rogovin,
and N. Shanmugalingam,
\textit{Measurability of equivalence classes and $MEC_p$-property in metric spaces},
Rev. Mat. Iberoam. 23 (2007), no. 3, 811--830.

\bibitem{KKLS}J. Kinnunen, R. Korte, A. Lorent, and N. Shanmugalingam,
\textit{Regularity of sets with quasiminimal boundary surfaces in metric spaces},
J. Geom. Anal. 23 (2013), no. 4, 1607--1640. 

\bibitem{KKST3}J. Kinnunen, R. Korte, N. Shanmugalingam, and H. Tuominen,
\textit{Pointwise properties of functions of bounded variation in metric spaces},
Rev. Mat. Complut. 27 (2014), no. 1, 41--67.

\bibitem{L-Fed}P. Lahti,
\textit{A Federer-style characterization of sets of finite perimeter on metric spaces},
 Calc. Var. Partial Differential Equations 56 (2017), no. 5, Art. 150, 22 pp.

\bibitem{L-SS}P. Lahti,
\textit{Capacities and 1-strict subsets in metric spaces},
preprint 2019.
https://arxiv.org/abs/1903.04358

\bibitem{L-Fedchar}P. Lahti,
\textit{Federer's characterization of sets of finite perimeter in metric spaces},
to appear in Analysis \& PDE.\\
https://arxiv.org/abs/1804.11216

\bibitem{L-SA}P. Lahti,
\textit{Strong approximation of sets of finite perimeter in metric spaces},
Manuscripta Math. 155 (2018), no. 3-4, 503--522.

\bibitem{L-WC}P. Lahti,
\textit{Superminimizers and a weak Cartan property for $p=1$ in metric spaces},
to appear in J. Anal. Math.

\bibitem{M}M.~Miranda, Jr.,
\textit{Functions of bounded variation on ``good'' metric spaces},
J. Math. Pures Appl. (9) 82  (2003),  no. 8, 975--1004.

\bibitem{Roy}H. L. Royden,
\textit{Real Analysis},
Third edition. Macmillan Publishing Company, New York, 1988. xx+444 pp.

\bibitem{S2}N. Shanmugalingam,
\textit{Harmonic functions on metric spaces},
Illinois J. Math. 45 (2001), no. 3, 1021--1050.

\bibitem{S}N. Shanmugalingam,
\textit{Newtonian spaces: An extension of Sobolev spaces to metric measure spaces},
Rev. Mat. Iberoamericana 16(2) (2000), 243--279.

\bibitem{Zie89}W. P. Ziemer,
\textit{Weakly differentiable functions. Sobolev spaces and functions of bounded variation},
Graduate Texts in Mathematics, 120. Springer-Verlag, New York, 1989. 

\end{thebibliography}
\end{document}